\PassOptionsToPackage{linktocpage=true}{hyperref}

\documentclass[12pt]{alt2021} 

\usepackage{enumitem}
\usepackage{amssymb}
\usepackage{amsmath}

\usepackage{mathrsfs}		
\usepackage{dsfont}

\usepackage{url}					
\hypersetup{colorlinks=true, urlcolor=blue}
\usepackage{xargs}
\usepackage{booktabs}

\usepackage{hhline}
\usepackage{multirow}

\usepackage[utf8]{inputenc} 
\usepackage[T1]{fontenc}

\usepackage{multicol}
\usepackage{silence}
\usepackage{amsfonts,thmtools}
\usepackage{mathtools}
\usepackage{xparse}
\usepackage{enumitem} 
\usepackage{etoolbox}
\usepackage{mathtools}
\usepackage{complexity}
\usepackage{svg}
\usepackage{xcolor, soul}
\usepackage{tablefootnote}
\usepackage[bb=boondox]{mathalpha} 
\definecolor{lightgrey}{rgb}{0.9,0.9,0.9}
\sethlcolor{lightgrey}

\definecolor{mygray}{rgb}{0.6,0.6,0.6}

\usepackage[capitalise,nameinlink,noabbrev]{cleveref}

\newtheorem{theorem}{Theorem}
\newtheorem{lemma}[theorem]{Lemma}
\newtheorem{proposition}[theorem]{Proposition}

\let\temp\epsilon
\let\epsilon\varepsilon
\let\varepsilon\temp

\usepackage{tikz}			
\usetikzlibrary{automata} 	
\usetikzlibrary{arrows} 	

\usetikzlibrary[shapes.arrows]
\usetikzlibrary{shapes.geometric}
\usetikzlibrary{backgrounds}
\usetikzlibrary{positioning}
\usetikzlibrary{calc}
\usetikzlibrary{intersections}
\usetikzlibrary{fadings}
\usetikzlibrary{decorations.footprints}
\usetikzlibrary{patterns}
\usetikzlibrary{shapes.callouts}
\usetikzlibrary{fit}

\tikzset{->, >=stealth', shorten >=1pt, auto, node distance=1cm, semithick, baseline=(current bounding box.center)}

\newcommand*\circledaux[1]{\tikz[baseline=(char.base)]{
    \node[shape=circle,draw,inner sep=0.8pt] (char) {#1};}}

\NewDocumentCommand{\circled}{ m o }{%
    \IfNoValueTF{#2}{ \circledaux{#1} }{ \stackrel{\circledaux{#1}}{#2} }%
}

\let\temp\dot 
\renewcommand{\dot}[1]{\langle #1\rangle}

\usepackage{times}

\title{Fast Algorithms for Packing Proportional Fairness and its Dual}

\author{%
  Francisco Criado\\
  TU Berlin \\ Berlin, Germany\\%
  \texttt{\href{mailto:criado@math.tu-berlin.de}{criado@math.tu-berlin.de}}
  \And
    David Martínez-Rubio\\
  Zuse Institute Berlin and TU Berlin\\ Berlin, Germany\\
  \texttt{\href{mailto:martinez-rubio@zib.de}{martinez-rubio@zib.de}} 
  \And
    Sebastian Pokutta\\
  Zuse Institute Berlin and TU Berlin\\ Berlin, Germany\\
  \texttt{\href{mailto:pokutta@zib.de}{pokutta@zib.de}} \\
}

\usepackage[natbib, backend=biber, maxcitenames=3, maxbibnames=99, style=alphabetic, hyperref, backref, useprefix=true, uniquename=false, doi=false,url=false,eprint=false]{biblatex} 

\usepackage{csquotes}               
\bibliography{refs}        

\DeclareCiteCommand{\cite}
  {\usebibmacro{prenote}}
  {\usebibmacro{citeindex}%
   \printtext[bibhyperref]{\usebibmacro{cite}}}
  {\multicitedelim}
  {\usebibmacro{postnote}}

\DeclareCiteCommand*{\cite}
  {\usebibmacro{prenote}}
  {\usebibmacro{citeindex}%
   \printtext[bibhyperref]{\usebibmacro{citeyear}}}
  {\multicitedelim}
  {\usebibmacro{postnote}}

\DeclareCiteCommand{\parencite}[\mkbibparens]
  {\usebibmacro{prenote}}
  {\usebibmacro{citeindex}%
    \printtext[bibhyperref]{\usebibmacro{cite}}}
  {\multicitedelim}
  {\usebibmacro{postnote}}

\DeclareCiteCommand*{\parencite}[\mkbibparens]
  {\usebibmacro{prenote}}
  {\usebibmacro{citeindex}%
    \printtext[bibhyperref]{\usebibmacro{citeyear}}}
  {\multicitedelim}
  {\usebibmacro{postnote}}

\DeclareCiteCommand{\citeauthor}
  {\usebibmacro{prenote}}
  {\ifciteindex
     {\indexnames{labelname}}
     {}%
   \printtext[bibhyperref]{\printnames{labelname}}}
  {\multicitedelim}
  {\usebibmacro{postnote}}

\DeclareCiteCommand{\footcite}[\mkbibfootnote]
  {\usebibmacro{prenote}}
  {\usebibmacro{citeindex}%
  \printtext[bibhyperref]{ \usebibmacro{cite}}}
  {\multicitedelim}
  {\usebibmacro{postnote}}

\DeclareCiteCommand{\footcitetext}[\mkbibfootnotetext]
  {\usebibmacro{prenote}}
  {\usebibmacro{citeindex}%
   \printtext[bibhyperref]{\usebibmacro{cite}}}
  {\multicitedelim}
  {\usebibmacro{postnote}}

\DeclareCiteCommand{\textcite}
  {\boolfalse{cbx:parens}}
  {\usebibmacro{citeindex}%
   \printtext[bibhyperref]{\usebibmacro{textcite}}}
  {\ifbool{cbx:parens}
     {\bibcloseparen\global\boolfalse{cbx:parens}}
     {}%
   \multicitedelim}
  {\usebibmacro{textcite:postnote}}

\newbibmacro{string+doiurlisbn}[1]{%
  \iffieldundef{doi}{%
    \iffieldundef{url}{%
      \iffieldundef{isbn}{%
        \iffieldundef{issn}{%
          #1%
        }{%
          \href{http://books.google.com/books?vid=ISSN\thefield{issn}}{#1}%
        }%
      }{%
        \href{http://books.google.com/books?vid=ISBN\thefield{isbn}}{#1}%
      }%
    }{%
      \href{\thefield{url}}{#1}%
    }%
  }{%
    \href{https://doi.org/\thefield{doi}}{#1}%
  }%
}

\DeclareFieldFormat{title}{\usebibmacro{string+doiurlisbn}{\mkbibemph{#1}}}
\DeclareFieldFormat[article,incollection,inproceedings]{title}%
    {\usebibmacro{string+doiurlisbn}{\mkbibquote{#1}}}

\usepackage{algorithm}
\usepackage{algcompatible}
\algnewcommand{\lst}{\texttt{lst}}
\algnewcommand{\slst}{\texttt{slst}}
\algnewcommand{\SEND}{\textbf{send}}

\newsavebox{\algleft}
\newsavebox{\algright}

\renewcommand{\algorithmicensure}{\textbf{Output:}}

\makeatletter
\newcounter{algorithmicH}
\let\oldalgorithmic\algorithmic
\renewcommand{\algorithmic}{%
  \stepcounter{algorithmicH}
  \oldalgorithmic}
\renewcommand{\theHALG@line}{ALG@line.\thealgorithmicH.\arabic{ALG@line}}
\makeatother



\input{definitions_links_global.tex}

\title[Fast Algorithms for Packing Proportional Fairness and its Dual]{Fast Algorithms for Packing Proportional Fairness and its Dual}

\altauthor{%
    \Name{Francisco Criado} \Email{\href{mailto:criado@math.tu-berlin.de}{criado@math.tu-berlin.de}}\\
 \addr Institute of Mathematics, Technische Universität \\ Berlin, Germany%
 \AND
 \Name{David Martínez-Rubio} \Email{\href{mailto:martinez-rubio@zib.de}{martinez-rubio@zib.de}}\\
 \addr Zuse Institute Berlin and Technische Universität\\ Berlin, Germany%
 \AND
 \Name{Sebastian Pokutta} \Email{\href{mailto:pokutta@zib.de}{pokutta@zib.de}}\\
 \addr Zuse Institute Berlin and Technische Universität\\ Berlin, Germany%
}
\hypersetup{pdfborder={0 0 0}}

\begin{document}

\maketitle

\begin{abstract}%
    The proportional fair resource allocation problem is a major problem studied in flow control of networks, operations research, and economic theory, where it has found numerous applications.
    This problem, defined as the constrained maximization of $\sum_i \log x_i$, is known as the packing proportional fairness problem when the feasible set is defined by positive linear constraints and $x \in \Rp^{\n}$. In this work, we present a distributed accelerated first-order method for this problem which improves upon previous approaches. We also design an algorithm for the optimization of its dual problem. Both algorithms are \textit{width-independent}. Finally, we show the latter problem has applications to the volume reduction of bounding simplices in an old linear programming algorithm of \citep{yamnitsky1982}, and we obtain some improvements as a result.
\end{abstract}

\footnotetext{Most of the notations in this work have a link to their definitions. For example, if you click or tap on any instance of $\ecanonical[i]$, you will jump to the place where it is defined as the $i$-th vector of the canonical base.}

\section{Introduction}\label{sec:intro}
The assignment of bounded resources to several agents under some notions of fairness is a topic studied in networking, operations research, game theory, and economic theory. The allocation obtained by the maximization of the function $\sum_{i=1}^{\n} \log(x_i)$ over a convex set $C\subseteq \Rp^{\n}$, known as a \emph{proportional fair allocation}, is an important solution that arises under a natural set of fairness axioms \citep{bertsimas2011, kao2010axiomatic}. It corresponds to Nash bargaining solutions \citep{nash1950bargaining} and it also has applications to multi-resource allocation in compute clusters \citep{bonald2015multiresource, jin2018tradeoff, joewong2012multiresource}, rate control in networks \citep{kelly1997charging} and game theory \citep{jain2010eisenberg, jain2007eisenberg}. Other important allocations are linear objectives (no fairness), the max-min allocations \citep{mo2000}, or $\alpha$-fair allocations \citep{atkinson1970measurement, mo2000, mccormick2014real}, which generalize all of the others. Proportional fairness corresponds to $\alpha=1$. A natural restriction, that many of these applications require, are positive linear constraints. This results in the \emph{packing proportional fairness problem}, also known as the \emph{$1$-fair packing problem}. The main focus of this paper is on solving this problem and its dual via first-order methods. Given $A \in \mathcal{M}_{\newtarget{def:m}{\m}\times \newtarget{def:n}{\n}}(\Rp)$, the $1$-fair packing problem is
\begin{align}\label{eq:primal_problem}
    \tag{1FP} 
    \max_{x\in \Rp^{\n}} \left\{ \newtarget{def:f}{\f}(x) \defi \sum_{i=1}^{\n} \log x_i : Ax \leq \ones_{\m} \right\}.
\end{align}
We also study the optimization of its Lagrange dual, that can be formulated, cf. \cref{lem:lagrange}, as
\begin{align}\label{eq:dual_problem}
	\tag{1FP-Dual}
    \min_{\lambda \in \simplex{\m}} \left\{ \newtarget{def:g}{\g}(\lambda) \defi -\sum_{i=1}^{\n} \log (A^T \lambda)_i - \n \log \n  \right\},
\end{align}
where $\newtarget{def:simplex}{\simplex{\m}}\defi \{\lambda \in \R^{\m} : \sum \lambda_i = 1, \lambda \geq 0 \}$ is the $\m$-dimensional (probability) simplex. We focus on \textit{width-independent} algorithms that  additively $\newtarget{def:epsilon}{\epsilon}$-approximate the optimum of those problems. For the $1$-fair packing problem that means, respectively, that we can find $\bar{x}$ in time that depends at most polylogarithmically on the width $\rho$ of the matrix $A$, and that it satisfies $\f^\ast-\f(\bar{x}) \leq \epsilon$, where $\f^\ast$ is the optimal value. Note that~\eqref{eq:primal_problem} has a unique optimizer, by strong concavity. By the same reason, for every two minimizers $\lambda_1$, $\lambda_2$ of \eqref{eq:dual_problem}, we have $A^T\lambda_1 = A^T\lambda_2$. The width $\newtarget{def:rho}{\rho}$ of $A$ is defined as $\max\{A_{ij}\}/\min_{A_{ij}\neq 0} \{A_{ij}\}$, the maximum ratio of the non-zero entries of $A$.
Note that in general width-dependent algorithms are not polynomial. 
Smoothness and Lipschitz constants of the objectives do not scale polylogarithmically with $\rho$ and thus, direct application of classical first-order methods leads to non-polynomial algorithms. As in packing and covering \LP{}, an approximate solution for our primal problem does not necessarily yield one for the dual problem, cf. \citep{awerbuch2008stateless}, so we need to study them separately. The current form of our techniques does not generalize to $\alpha$-fair problems with $\alpha\neq 1$, but generalizing them to these settings is an interesting future direction of research. We note that previous works treat $\alpha$ in $[0,1)$, $\{1\}$, or $(1,\infty)$ separately, due to the structure of the problems being different.  Most works dealing with $\alpha$-fair functions assume, without loss of generality, that $A$ is given so that the minimum non-zero entry of $A$ is $1$ and the maximum entry is $\rho$. However, in this work, we assume without loss of generality that
\begin{equation} \label{eq:normalization_of_A}
    \max_{i\in[\m]}\{ A_{ij} \} = 1,\text{ for all }j\in[\n].
\end{equation}
We can do so because, for our problem, we can rescale each primal coordinate multiplicatively, rescaling the columns of $A$ accordingly, which only changes the objectives by an additive constant. Thus, the additive guarantees we will obtain are also satisfied in the non-scaled problem.

Our primal algorithm solves the problem in a distributed model of computation with $\n$ agents. Each agent $j\in [\n]$ controls variable $x_j$ and only has access to global parameters like $\m,\n,$ or the target accuracy $\epsilon$, to the $j$-th column of $A$, and in each round it receives the slack $(Ax)_i - 1$ of all the constraints $i$ in which $j$ participates. This is a standard distributed model of computation. We refer to \citep{kelly2014stochastic,awerbuch2008stateless} for its motivation and applications.

\paragraph{Notations}
We let $\newtarget{def:e_i_canonical_basis}{\ecanonical[i]}$ be the vector with $1$ in coordinate $i$ and $0$ elsewhere. We denote by $A_i$ a row of $A$. For $k\in \mathbb{N}$, we use the notation $[k] \defi \{1, 2, \dots, k\}$. Throughout this work, $\log(\cdot)$ represents the natural logarithm. For $v\in\R^{\n}$, the notation $\exp(v)$ means entrywise exponential. We use $\newtarget{def:entrywise_prod}{\odot}$ for the entrywise product. Given a $1$-strongly convex map $\psi$, we denote its Bregman divergence by $\newtarget{def:bregman_div}{\breg}(x,y)\defi \nabla \psi(x)-\nabla \psi(y)-\innp{\nabla \psi(y), x-y}$. We denote by $\newtarget{def:N}{\N}$ the number of non-zero entries of the matrix $A$. The notation $\newtarget{def:big_o_tilde}{\bigotilde{\cdot}}$ omits logarithmic factors with respect to $m$, $n$, $1/\epsilon$ and $\rho$. But note that the rates of our algorithms do not depend on $\rho$.

\paragraph{Related Work}
Despite the importance and widespread applicability of fairness objectives, width-independent (and thus polynomial) algorithms for many $\alpha$-fair packing problems were not developed until recently. Width-independent algorithms were first designed for $0$-fair packing, i.e., for packing linear programming (\newtarget{def:acronym_linear_programming}{\LP{}}), that have a longer history \citep{luby1993parallel}. For this problem there are currently nearly linear-time width-independent iterative algorithms \citep{allen2019nearly} and distributed algorithms \citep{allen2015using, diakonikolas2017solving}. \citep{marasevic2015fast} studied the width-independent optimization of $\alpha$-fair packing problems for any $\alpha \in [0, \infty]$ with a stateless algorithm and \citep{diakonikolas2020fair} gave better rates with a non-stateless algorithm. Both works use the same distributed framework as ours. For the particular case of $1$-fair packing, the latter work obtains an unaccelerated algorithm that runs in $\bigotilde{\n^2/\epsilon^2}$ distributed iterations. \citep{beck2014gradient} study the optimization of the dual problem by using Nesterov's accelerated method, and then they reconstruct a primal solution. However, both primal and dual solutions depend on the smoothness constant of the dual problem, which in the worst case is proportional to $\rho^2$, and therefore it is not a polynomial algorithm. In contrast, our algorithms do not depend on $\rho$ at all. Obtaining a priori lower bounds on each of the coordinates of the optimizer is of theoretical and practical interest, since it provides certain amount of resource that can be assigned to each agent before solving the problem. These were studied in \citep{marasevic2015fast} and were improved by \citep{allybokus2018lower}. In \cref{lemma:lower_bound_on_coordinates_of_x_ast}, we show a lower bound of this kind for our problem when it is normalized as in \eqref{eq:normalization_of_A}. 

\begin{table}
    \centering
    \caption{Comparison of algorithms for $1$-fair packing and its dual. The work of one iteration is linear in $\N$, the number of non-zero entries in $A$.}
    \label{table:comparisons}

\begin{tabular}{lclc}
    \toprule
    \textbf{Paper} & \textbf{Problem} & \textbf{Iterations} & \textbf{Width-dependence?} \\
    
    \midrule
    \midrule
    \citep{beck2014gradient}        & Primal & $\bigo{\rho^{2}\m\n/\epsilon}$         & Yes\\
    \citep{marasevic2015fast}       & Primal & $\bigotilde{\n^5/\epsilon^5}$ & \ \ \ \ \ {\color{mygray}nearly} No {\color{mygray}(polylog)} \\
    \citep{diakonikolas2020fair}    & Primal & $\bigotilde{\n^2/\epsilon^2}$ & \ \ \ \ \ {\color{mygray}nearly} No {\color{mygray}(polylog)}\\
    \textbf{This paper} (\cref{thm:primal_guarantee})     & Primal & $\bigotilde{\n/\epsilon}$  & No \\
    \midrule
    \citep{beck2014gradient}        & Dual   & $\bigo{\rho\sqrt{\m\n/\epsilon}}$    & Yes\\
    \textbf{This paper} (\cref{thm:dual_guarantee})       & Dual   & $\bigotilde{\n^2/\epsilon}$ & No \\
    \bottomrule
\end{tabular}

\end{table}

\paragraph{Contribution and Main Results}
Our contribution can be summarized as follows; See \cref{table:comparisons} for a comparison with previous works.\medskip

\emph{Accelerated algorithm for $1$-fair packing.} We design a distributed accelerated algorithm for $1$-fair packing by generalizing and extending an accelerated technique, designed for packing \LP{}, that uses truncated gradients of a regularized objective \citep{allen2019nearly}. In contrast with this technique, ours yields an algorithm and guarantees that are deterministic. We exploit the structure of our problem to obtain a distributed solution, while for packing \LP{} obtaining a distributed or just parallel algorithm that is accelerated and width-independent is an open question \citep{diakonikolas2017solving}. We make use of a different regularization and an analysis that yields additive error guarantees as opposed to multiplicative ones. 

\emph{The dual problem.} We consider the dual of the $1$-fair packing problem. We reduce the problem to optimizing a proxy function by using the Plotkin-Shmoys-Tardos (\newtarget{def:acronym_plotkin_shmoys_tardos}{\PST{}}) framework \citep{plotkin1995fast, arora2012multiplicative} with a novel geometric separation oracle. Critical to obtaining fast convergence is showing that the oracle parameters decrease when we obtain better solutions. This fact allows to reduce the dependence on $\epsilon$, and as a result, our width-independent algorithm enjoys a convergence rate of $\bigotilde{\n^2/\epsilon}$ iterations for this problem. 

\emph{Algorithm for the minimum-volume simplex problem.} Finally, we use the log-volume interpretation to present a new application of the dual problem to the approximation of the simplex $\Delta^{(k+1)}$ of minimum volume that covers a polytope $\P$, where $\Delta^{(k+1)}$ is given by a previous bounding simplex $\Delta^{(k)}$ containing $\P$, and where exactly one facet is allowed to move. This results in some improvements to the old \emph{method of simplices} algorithm by \citep{yamnitsky1982} for \LP{}.

\section{A distributed accelerated algorithm for 1-Fair Packing} \label{sec:primal}
{


\let\oldbeta\beta
\renewcommand\beta{\newlink{def:beta}{\oldbeta}}
\let\oldomega\omega
\renewcommand\omega{\newlink{def:omega}{\oldomega}}
\let\oldtau\tau
\renewcommand\tau{\newlink{def:tau}{\oldtau}}

\newcommandx*\truncgrad[2][1=j, 2=\fr, usedefault]{\newlink{def:truncgrad}{\overline{\nabla_{#1} #2}}}

\newcommandx*\nuk[2][1=k, 2= , usedefault]{\newlink{def:remainder_trunc_coord_grad}{\nu}^{( #1 )}_{#2}}

\renewcommandx*\x[2][1=k, 2= , usedefault]{\newlink{def:xk_algorithm}{x^{( #1 )}_{#2}}}
\renewcommandx*\y[2][1=k, 2= , usedefault]{\newlink{def:yk_algorithm}{y^{( #1 )}_{#2}}}
\renewcommandx*\z[2][1=k, 2= , usedefault]{\newlink{def:zk_algorithm}{z^{( #1 )}_{#2}}}

\newcommandx*\Ck[2][1=k, 2= , usedefault]{\newlink{def:Ck}{C_{#1}^{#2} }}
\newcommandx*\etak[2][1=k, 2= ,  usedefault]{\newlink{def:eta_k}{\eta_{#1}^{#2} }}

\newcommandx\fr{\newlink{def:f_r}{f_r}}
\newcommandx\fhat{\newlink{def:f_hat}{\hat{f}}}

\renewcommand\L{\newlink{def:L}{L}}

\newcommand\B{\newlink{def:Box}{B}}

\newcommand\xrast{\newlink{def:xrast}{x_r^\ast}}
\newcommand\xasthat{\newlink{def:xasthat}{\hat{x}^\ast}}

\newcommandx*\lossk[2][1=k, 2= , usedefault]{\newlink{def:losses_primal}{\ell^{(#1)}_{#2}}}

In this section, we present the main steps of our algorithm for the primal problem, which is a deterministic accelerated descent method that optimizes an objective coming from the $1$-fair packing problem, and that encodes the constraints in the form of a barrier. Our algorithm approximates the objective additively and allows to compute each iteration in a distributed manner. We note that \citep{diakonikolas2020fair} also made use of this intermediate objective for the $1$-fair packing problem with different constants, but as opposed to their solution, we allow to compute unfeasible solutions during the course of the algorithm, and we proceed with different techniques that allow to achieve acceleration and thus an algorithm with better convergence rates. We defer some proofs to \cref{app:proofs_of_primal_section}.

We reparametrize Problem~\eqref{eq:primal_problem} so that the objective function is linear at the expense of making the constraints more complex. That is, we define the function $\newtarget{def:f_hat}{\fhat}:\R^{\n} \to \R, x\mapsto \f(\exp(x)) = \innp{\ones_{\n}, x}$. The optimization problem becomes
\begin{equation}\label{eq:reparametrization_1_fair_packing}
    \max_{x\in \R^{\n}}\left\{\fhat(x) \defi \innp{\ones_{\n}, x}: A\exp(x) \leq \ones_{\m}\right\}.
\end{equation}
Then, we regularize the negative of the reparametrized objective by adding a fast-growing barrier:
    \begin{align*}
        \newtarget{def:f_r}{\fr}(x) &\defi -\innp{\ones_{\n}, x} + \frac{\beta}{1+\beta}\sum_{i=1}^{\m} (A\exp(x))_i^{\frac{1+\beta}{\beta}}, \nabla_j \fr(x) &= -1  + \sum_{i=1}^{\m} (A\exp(x))_i^{\frac{1}{\beta}} a_{ij} \exp(x_j), 
    \end{align*}
where $\newtarget{def:beta}{\beta} \defi \frac{\epsilon}{6\n \log(2\m\n^2/\epsilon)}$. In this way, we can work with an unconstrained minimization problem. The resulting function is not globally smooth but when the absolute value of a coordinate of the gradient is large, it is positive, and in that case we are able to take a small gradient descent step and decrease the function considerably. The intuition is that if the gradient is large, then the function value along the segment of the gradient step, as a function of the step, can decrease fast. But it cannot increase fast since there are no large negative gradient coordinates. We depict $\fr$ in \cref{fig:barrier_intuition}. The barrier also allows to maintain almost feasibility, as we show in \cref{prop:optimizing_the_smoothened_function_is_enough} below. It is chosen to grow fast enough so that a point satisfying $(A\exp(x))_i > 1+\epsilon/\n$, for some $i\in[\n]$, will have an optimality gap that is greater than the required accuracy. On the other hand, the regularizer is very small in the feasible region that is not too close to the boundary.

\begin{figure}
    \centering
    \begin{minipage}{0.5\textwidth}
        
        \includegraphics[width=\textwidth]{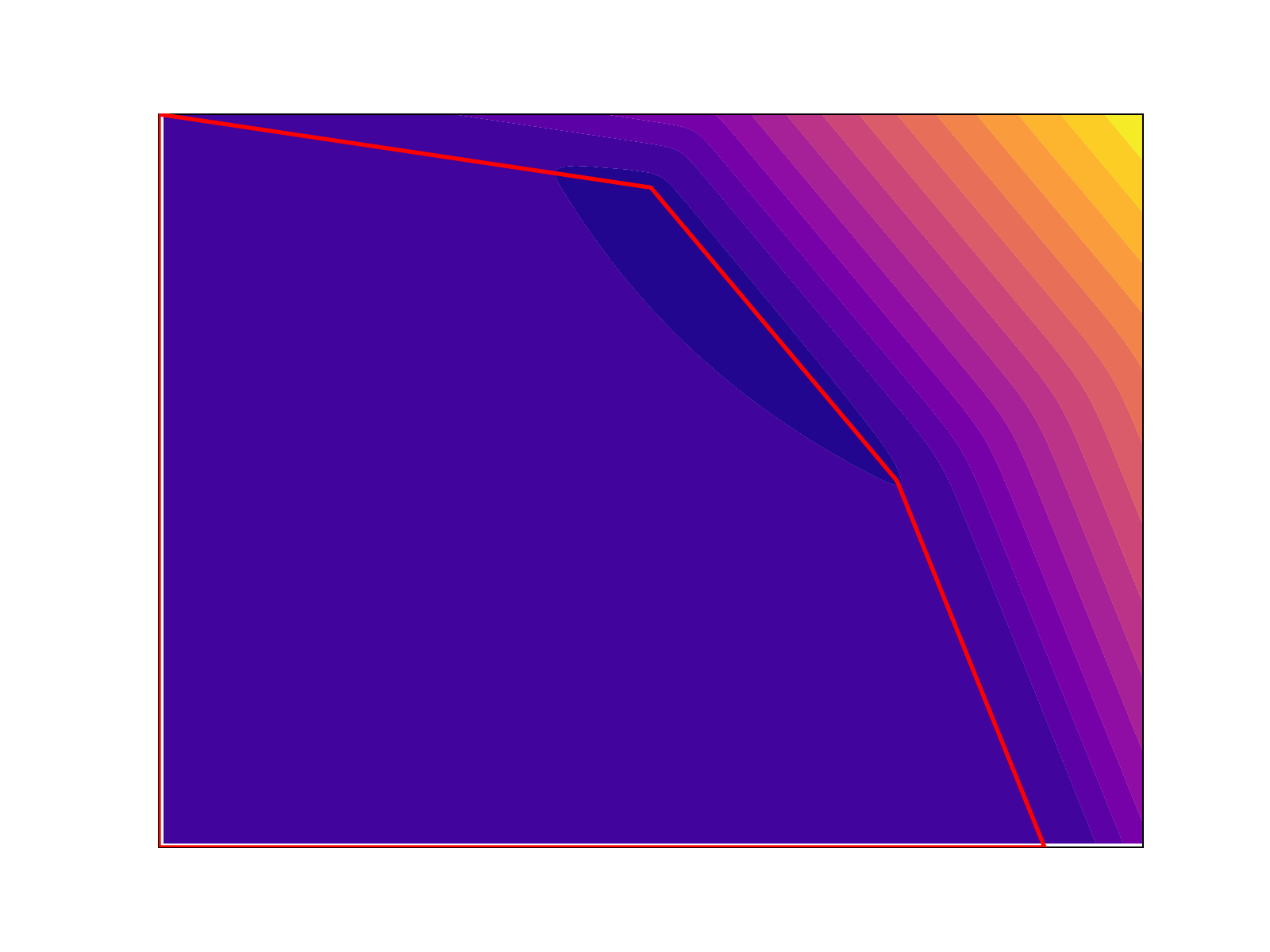}
        
    \end{minipage}\hfill
    \begin{minipage}{0.5\textwidth}
        
        \includegraphics[width=\textwidth]{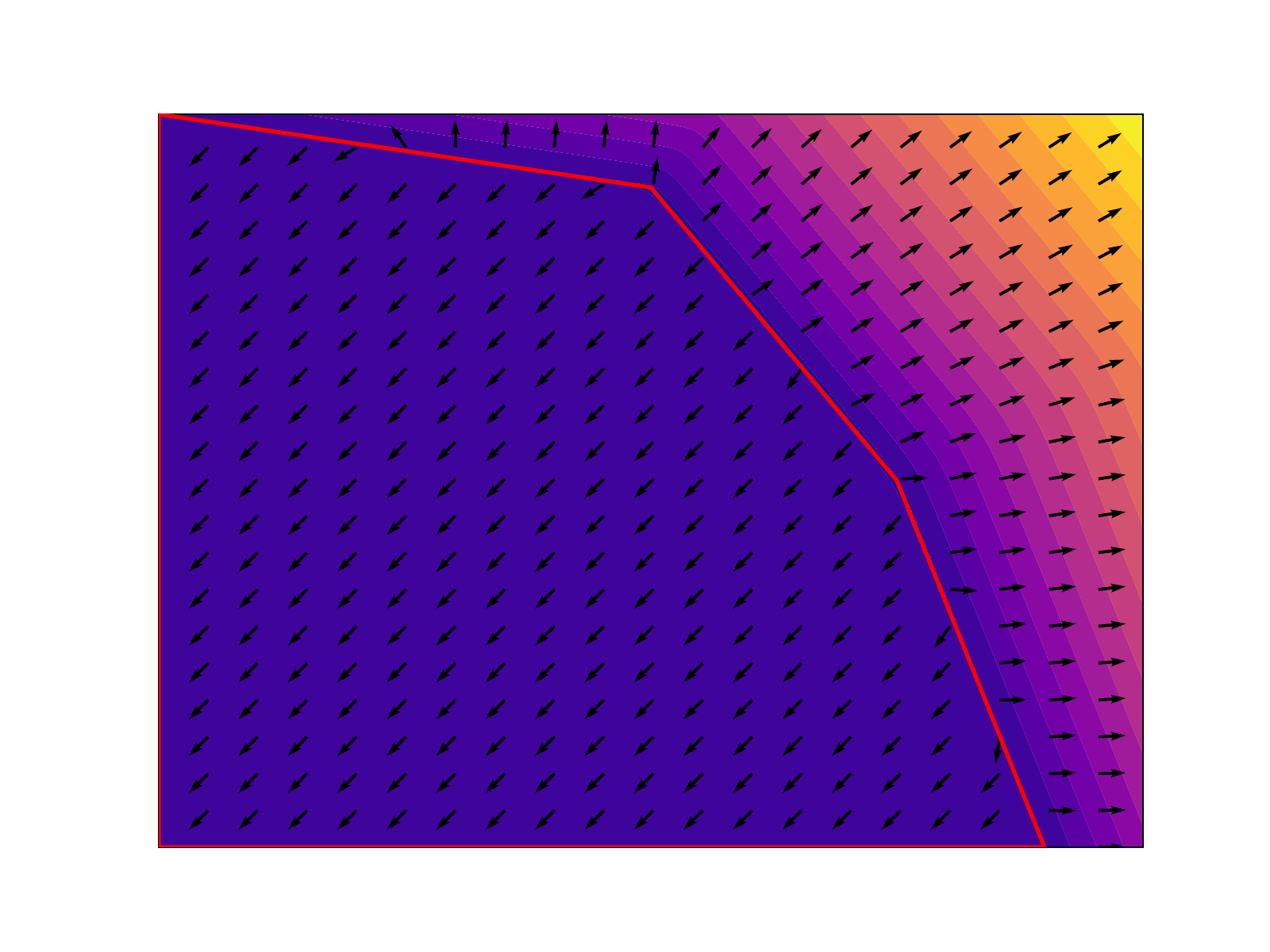}
        
    \end{minipage}
        \caption{Regularized objective $\fr$ (left) and its gradient (right), for a sample matrix $A\in\mathcal{M}_{3\times 2}(\Rp)$. For visualization purposes we show $\log (\fr(x))$ and $\log(\norm{\nabla \fr(x)})$, represented by color, and we indicate the direction of the gradient with normalized arrows. Also, note that we show the results in the original space (i.e., before reparametrizing, so the constraints appear to be linear) but the gradient was computed as originally defined (i.e., after reparametrizing).} 
        \label{fig:barrier_intuition}
\end{figure}

    Let $\newtarget{def:xasthat}{\xasthat}$ be the maximizer of $\fhat$, let $\newtarget{def:xast}{\xast} \defi \exp(\xasthat)$ be the solution to Problem~\eqref{eq:primal_problem}, and let $\newtarget{def:xrast}{\xrast}$ be the minimizer of $\fr$. We have $\xasthat \in [-\log(\n), 0]^{\n}$ by \cref{lemma:lower_bound_on_coordinates_of_x_ast}. Let $\newtarget{def:omega}{\omega} \defi \log(\m\n/(1-\epsilon/\n))$ and define the box $\newtarget{def:Box}{\B} \defi [-\omega, 0]^{\n}$. We restrict ourselves to this domain and formulate our final problem, that we will minimize with an accelerated method:

\begin{align}\label{eq:regularized_reparametrized_1_fair_packing}
    \tag{1FP-primalReg}
    \min_{x\in \B}\fr(x).
\end{align}
    Note $\fr(x) \geq 0$ if $x\in \B$. We add the redundant and simple box constraints $\B$ in order to later guarantee a bound on the regret of the mirror descent method that runs within the algorithm. We show that it suffices to obtain an $\epsilon$-minimizer of Problem~\eqref{eq:regularized_reparametrized_1_fair_packing} in order to obtain an $O(\epsilon)$-minimizer for the original Problem~\eqref{eq:primal_problem}.

\begin{proposition}\label{prop:optimizing_the_smoothened_function_is_enough} \linktoproof{prop:optimizing_the_smoothened_function_is_enough}
    Let $\epsilon\in (0, \n/2]$. Let $\xrast$ be the minimizer of \eqref{eq:regularized_reparametrized_1_fair_packing} and let $x_r^{\epsilon}\in \B$ be an $\epsilon$-minimizer of this problem. Then the point $\bar{u} \defi \exp(x_r^{\epsilon})/(1+\epsilon/\n)$ satisfies $\f(\xast)-\f(\bar{u}) \leq 5\epsilon$ and $A \bar{u} \leq \ones_{\m}$, where $\xast$ is the maximizer of $\f$. 
\end{proposition}

    The intuition about this proposition is that $x_r^{\epsilon}$ is also a point with low $\hat{f}$ value. By the aforementioned barrier guarantees, it is almost feasible, i.e., $A\exp(x_r^{\epsilon}) \leq 1+\epsilon/n$, and dividing the corresponding $\exp(x_r^{\epsilon})$ by $1+\epsilon/n$, and thus making it feasible, can only increase the objective $f$ by $\epsilon$. 

    In the sequel, we will present the different parts of \cref{alg:1_fair_packing} and their analyses. In particular, the notation and definitions used are compatible with the choices in the algorithm and most of the parameter choices naturally occur throughout the arguments.  Our optimization algorithm starts at the points $\x[0] = \y[0] = \z[0] = -\log(\m\n/(1-\epsilon/\n)) \ones_{\n}$ and updates each of these variables $\x, \y$ and $\z$ once in each iteration. They remain in $\B$, by \cref{lemma:iterates_remain_in_B}. The role of the three variables is the following: $\z$ will be a mirror point and $\y$ will be a gradient descent point, in the sense that in order to compute them we apply mirror descent and gradient descent. Then, the point $\x$ will be a convex combination of both, that will balance the regret of $\z$ with the primal progress of $\y$, effectively coupling these two algorithms.

\begin{lemma}\linktoproof{lemma:iterates_remain_in_B} \label{lemma:iterates_remain_in_B}
    The iterates of \cref{alg:1_fair_packing} remain in the box $\B$.
\end{lemma}

    It is important to note that we do not use the gradient $\nabla \fr(x)$ for our mirror descent loss. Instead, we use a truncation of the gradient. More precisely, the loss we perform the mirror descent step on is the truncated gradient $\newtarget{def:truncgrad}{\truncgrad[ ]}(\x) \in \R^{\n}$ defined as
    \begin{equation}\label{eq:def_of_truncated_gradient}
        \truncgrad[i](\x) \defi \min\{1, \nabla_{i} \fr(\x)\} \text{ for all }i\in[\n].
    \end{equation}
    Note that $\truncgrad[ ](\x) \in [-1, 1]^{\n}$ because $\nabla \fr(x) \in [-1, \infty]^{\n}$ for any $x \in \R^{\n}$, as the regularizer has positive gradient; see also definition of $\fr(x)$ and its gradient. The truncation allows mirror descent to control one part of the regret, which will not depend on the global Lipschitz constant. Gradient descent will compensate for both such regret and the part that is not controlled by mirror descent.

Let $\Pi_{\X}(\cdot)$ be the $\norm{\cdot}_2$-projection map of a point onto a convex set $\X$. The mirror descent update can be written in closed form as any of the two following equivalent ways
\begin{align} \label{eq:closed_form_mirror_descent_in_the_box}
       \begin{aligned}
       \z[] &\gets \Pi_{\B} (\z[k-1] - \omega\etak\truncgrad[ ](\x)),\\
       \z[][i] &\gets \Pi_{[-\omega, 0]} (\z[k-1][i] - \omega\etak\truncgrad[i](\x)), \text{ for all } i\in[\n].
       \end{aligned}
    \end{align}
That is, projecting back to the box, in case of the $\norm{\cdot}_2$, consists of simply clipping each coordinate. We bound the regret coming from this mirror descent step by modifying the classical analysis of mirror descent, cf. \cref{lemma:mirror_descent_lemma}.

\begin{lemma}[Mirror Descent Guarantee]\label{lemma:mirror_descent_guarantee}
    Let $u \in \B$ and choose $\L$ as in \cref{alg:1_fair_packing}. We have:
    \[
        \innp{\etak\truncgrad[ ](\x), \z[k-1]-u} \leq \etak[k][2] \L \innp{\truncgrad[ ](\x), \x-\y} + \frac{1}{2\omega} \norm{\z[k-1]-u}_2^2 - \frac{1}{2\omega}\norm{\z-u}_2^2.
    \]
\end{lemma}

\begin{proof}\label{proof:mirror_descent_guarantee}
    Use \cref{lemma:mirror_descent_lemma}.\ref{lemma:enum:mirror_descent_non_standard_lemma} below with loss $\lossk = \truncgrad[ ](\x)$, learning rate $\eta = \etak$, and regularizer $\psi(x) = \frac{1}{2\omega}\norm{x}_2^2$, that yields Bregman divergence $\breg(x, y) = \frac{1}{2\omega}\norm{x-y}_2^2$. Use that $\z[k-1]-\z = \etak\L(\x-\y)$.
\end{proof}

Next, we will analyze the role of the gradient descent step. We show in the following lemma a lower bound on the progress of our descent step. Note that this progress could not be greater than $\innp{\nabla \fr(\x), \x-\y}$, by convexity, so this is a strong descent condition. In the proof, we will use \cref{lemma:fair_packing_descent_lemma}, which is a crucial generalization of \citep[Lemma 3.1]{diakonikolas2020fair}.

\begin{lemma}[Descent Lemma]\label{lemma:descent_lemma_coord_descent} 
    Given $\x$ and $\y$ as defined in \cref{alg:1_fair_packing}, the following holds:
    \[
    \fr(\x) -\fr(\y) \geq \frac{1}{2}\innp{\nabla \fr(\x), \x-\y} \geq 0.
    \]
\end{lemma}

\begin{proof}
    We have $\x-\y = (\z[k-1]-\z)/\etak\L$ by definition of the gradient descent step. With this, we first conclude that  $\frac{1}{2}\innp{\nabla \fr(\x), \x-\y}\geq 0$, as $\nabla_{i} \fr(\x)$ and $\x[k][i]-\y[k][i]$ have the same sign for all $i\in[\n]$, cf. \eqref{eq:closed_form_mirror_descent_in_the_box}.

    We apply \cref{lemma:fair_packing_descent_lemma} with $\y$ corresponding to $x+\Delta$ and $\x$ corresponding to $x$. To this end, we choose $c_{i} \geq 0$ satisfying $\circled{1}$ below
    \[
        \frac{c_{i}\beta}{4(1+\beta)}\abs{\truncgrad[i](\x)} \circled{1}[=] \abs{\x[k][i]-\y[k][i]} \circled{2}[=] \frac{1}{\etak\L}\abs{\z[k-1][i]-\z[k][i]} \circled{3}[\leq] \frac{\omega}{\L}\abs{\truncgrad[i](\x)},
    \]
    where $\circled{2}$ holds by definition of $\y$ and $\circled{3}$ holds by the mirror descent update \eqref{eq:closed_form_mirror_descent_in_the_box}. Thus, it suffices to pick $c_{i}$ such that $c_{i} \leq \frac{4\omega(1+\beta)}{\beta \L} \leq 1$, where the last inequality holds true by the definition of $\L$. In fact, the value of $\L$ was chosen to satisfy the previous inequality. Hence, \cref{lemma:fair_packing_descent_lemma} can be applied. We obtain:
    \[
        \fr(\x) -\fr(\y) \geq \sum_{i=1}^{\n} \left(1-\frac{c_i}{2}\right)\nabla_i \fr(\x) (\x[k][i]-\y[k][i]) \geq \frac{1}{2}\innp{\nabla \fr(\x), \x-\y}.
    \]
    as desired.
\end{proof}

\begin{algorithm}
    \caption{Accelerated descent method for $1$-Fair Packing} 
    \label{alg:1_fair_packing}
\begin{algorithmic}[1]
    \REQUIRE Matrix $A \in \mathcal{M}_{\m\times \n}(\Rp)$ normalized as in \eqref{eq:normalization_of_A}. Accuracy $\epsilon \in (0, \n/2]$.
    \State $\beta \gets \frac{\epsilon}{6\n \log(2\m\n^2/\epsilon)}$; $\omega \gets \log(\frac{\m\n}{1-\epsilon/\n})$; $\newtarget{def:L}{\L} = \max\left\{\frac{4\omega(1+\beta)}{\beta}, \frac{16\n\log(2\m\n)}{3\epsilon} + \frac{1}{3}\right\} = \bigotilde{\n/\epsilon}$
    \State $\etak[0] \gets \frac{1}{3\L}$; $\Ck = 3\etak\L$; $\tau \gets \tau_k = \etak /\Ck = \frac{1}{3\L}$.
    \State $T \gets \lceil \log(\frac{4\n\log(2\m\n)}{\epsilon})/\log(\frac{1}{1-\tau}) \rceil \leq \lceil3\L\log(\frac{4\n\log(2\m\n)}{\epsilon})\rceil = \bigotilde{\n/\epsilon}$
    \State $\x[0] \gets \y[0] \gets \z[0] \gets -\omega \ones_{\n}$
    \vspace{0.1cm}
    \hrule
    \vspace{0.1cm}
    \FOR {$k = 1 \text{ \textbf{to} } T$}
        \State $\newtarget{def:eta_k}{\etak} \gets \Ck - \Ck[k-1] = \frac{1}{1-\tau}\etak[k-1]$
        \State $\newtarget{def:xk_algorithm}{\x} \gets \tau\z[k-1] + (1-\tau) \y[k-1]$
        \State $\newtarget{def:zk_algorithm}{\z} \gets \argmin_{z\in \B}\left\{ \frac{1}{2\omega}\norm{z-\z[k-1]}_2^2 + \innp{\etak\truncgrad[ ](\x), z}\right\}$ \Comment{Mirror descent step}
        \State \label{line:gradStep}$\newtarget{def:yk_algorithm}{\y} \gets \x + \frac{1}{\etak\L}(\z-\z[k-1]) $\Comment{Gradient descent step}
    \ENDFOR
    \State \textbf{return} $\bar{x} \defi \exp(\y[T])/(1+\epsilon/\n)$
    \ENSURE $\f(\bar{x} ) - \f(\xast) \leq \epsilon$ and $\bar{x} $ is feasible, i.e., $A\bar{x}  \leq 1$. The total number of iterations is $\bigotilde{\n/\epsilon}$ to obtain an $\bigo{\epsilon}$-approximate solution.
\end{algorithmic}
\end{algorithm}

\subsection{Coupling Mirror Descent and Gradient Descent}

We first prove a lemma that shows we can compensate for the regret coming from mirror descent as well as for the rest of the regret. Note the total weighted instantaneous regret $\innp{\etak\nabla \f(\x), \z[k-1]-u}$ is bounded by the left hand side of \eqref{eq:compensating_for_all_of_the_regret} up to a difference of potential functions, by \cref{lemma:mirror_descent_guarantee}. This is a critical part of the analysis: using the truncated gradient for mirror descent makes its corresponding regret not to depend on the smoothness constant, but there is a remaining regret that, crucially, can be compensated by our strong descent condition.

\begin{lemma}\label{lemma:gradient_descent_compensates_packing_regret}\linktoproof{lemma:gradient_descent_compensates_packing_regret}
    Let $\newtarget{def:Ck}{\Ck} \defi 3\etak \L$, and let $\newtarget{def:remainder_trunc_coord_grad}{\nuk} \defi \nabla \fr(\x)-\truncgrad[ ](\x)\in [0,\infty)^{\n}$. For all $u \in \B$, we have
    \begin{equation}\label{eq:compensating_for_all_of_the_regret}
        \innp{\etak\nuk, \z[k-1]-u} + \etak[k][2] \L \innp{\truncgrad[ ](\x), \x-\y} \leq \Ck (\fr(\x)-\fr(\y)).
    \end{equation}
\end{lemma}

With these tools at hand, we can now use a linear coupling argument to establish an accelerated convergence rate. Note that the algorithm takes the simple form of iterating a mirror descent step, gradient descent step and a coupling, after a careful choice of parameters. All of which depend on known quantities.

\begin{theorem}\label{thm:primal_guarantee}
    Let $\epsilon \leq \n/2$ and let $\xast$ be the solution to \eqref{eq:primal_problem} and let $\xrast$ be the minimizer of \eqref{eq:regularized_reparametrized_1_fair_packing}. \cref{alg:1_fair_packing} computes a point $\y[T]\in\B$ such that $\fr(\y[T])-\fr(\xrast) \leq \epsilon$ in a number of iterations $T =\bigotilde{\n/\epsilon}$. Besides, $\bar{x} \defi \exp(\y[T])/(1+\epsilon/\n)$ is a feasible point of \eqref{eq:primal_problem}, i.e., $A\bar{x} \leq \ones_{\m}$, such that $\f(\xast)-\f(\bar{x}) \leq 5\epsilon = O(\epsilon)$.
\end{theorem}

\begin{proof} We start by bounding the gap with respect to $\x$:
\begingroup
\allowdisplaybreaks
    \begin{align} \label{eq:coupling_inequalities_fair_packing}
       \begin{aligned}
           &\etak(\fr(\x)-\fr(u)) \\
           &\circled{1}[\leq] \innp{\etak \nabla \fr(\x), \x-u} \\
           &= \innp{\etak \nabla \fr(\x), \x-\z[k-1]} + \innp{\etak\nuk, \z[k-1]-u} + \innp{\etak\truncgrad[ ](\x), \z[k-1]-u} \\
           & \circled{2}[=] \frac{(1-\tau)\etak}{\tau} \innp{\nabla \fr(\x), \y[k-1] - \x)} + \innp{\etak\nuk, \z[k-1]-u} + \innp{\etak\truncgrad[ ](\x), \z[k-1]-u} \\
           & \circled{3}[\leq] \frac{(1-\tau)\etak}{\tau} (\fr(\y[k-1])-\fr(\x)) + \innp{\etak\nuk, \z[k-1]-u} + \innp{\etak[k][2]\L\truncgrad[ ](\x), \x-\y} \\ 
           &\quad \quad + \frac{1}{2\omega}\norm{\z[k-1]-u}_2^2 - \frac{1}{2\omega}\norm{\z-u}_2^2] \\
           & \circled{4}[\leq] \frac{(1-\tau)\etak}{\tau} (\fr(\y[k-1])-\fr(\x)) + \Ck(\fr(\x)-\fr(\y)) + \frac{1}{2\omega}\norm{\z[k-1]-u}_2^2 \\ 
           &\quad \quad  - \frac{1}{2\omega}\norm{\z-u}_2^2 \\
           & \circled{5}[\leq] \etak \fr(\x) + (\Ck-\etak)\fr(\y[k-1]) -\Ck\fr({\y}) + \frac{1}{2\omega}\norm{\z[k-1]-u}_2^2 - \frac{1}{2\omega}\norm{\z-u}_2^2 \\
       \end{aligned}
    \end{align}
\endgroup
    We used convexity in $\circled{1}$. The definition of $\x$ is used in $\circled{2}$. Inequality $\circled{3}$ uses convexity and \cref{lemma:mirror_descent_guarantee}. We applied \cref{lemma:gradient_descent_compensates_packing_regret} in $\circled{4}$. In $\circled{5}$, we substituted the value of $\tau$, which is picked to be $\newtarget{def:tau}{\tau} \defi \etak/\Ck = \frac{1}{3\L} $ so we can cancel $\fr(\x)$ in both sides of \eqref{eq:coupling_inequalities_fair_packing}.

The choice of $\etak$ is made so that $\Ck-\etak = \Ck[k-1]$ (or equiv. $(3\L-1)\etak=3\L\etak[k-1]$), which allows to telescope the previous expression. Adding up \eqref{eq:coupling_inequalities_fair_packing} for $k=1,\dots,T$ with $u = \xrast$, we have
\begin{align*}
   \begin{aligned}
       \left(-\Ck[0]-\sum_{k=1}^T\etak \right) \fr(\xrast) &\leq \Ck[0] (\fr(\y[0])-\fr(\xrast)) - \Ck[T] \fr(\y[T]) + \frac{1}{2\omega}\norm{\z[0]-\xrast}_2^2. \\
   \end{aligned}
\end{align*}
    We dropped $-\frac{1}{2\omega}\norm{\z[T]-\xrast}_2^2 \leq 0$. Now, since $\etak = \Ck - \Ck[k-1]$ we have $-\Ck[0]-\sum_{k=1}^T\etak = -\Ck[T]$. So reorganizing terms we obtain
\begin{align} \label{eq:final_convergence_inequality_1_fair}
   \begin{aligned}
       \fr(\y[T]) &\leq  \fr(\xrast) + \frac{1}{\Ck[T]}\left( \Ck[0] (\fr(\y[0])-\fr(\xrast))   + \frac{1}{2\omega}\norm{\z[0]-\xrast}_2^2 \right) \\
       &\circled{1}[\leq] \fr(\xrast) + \frac{1}{\Ck[T]}\left( \Ck[0] (\n(\log(2\m\n) + 1)  + \frac{\n\log(\m\n)}{2} \right) \\
       &\circled{2}[\leq] \fr(\xrast) + \epsilon
   \end{aligned}
\end{align}
    Above, $\circled{1}$ uses $\fr(\y[0]) \leq \n\log(\m\n/(1-\epsilon/\n)) + \epsilon \leq \n(\log(2\m\n) + 1)$ and $-\fr(\xrast) \leq 0$. For the former, take into account that $-\log(\m\n)\ones_{\n}$ is feasible and so the regularizer at $\y[0] = -\log(\m\n/(1-\epsilon/\n)) \ones_{\n}$ is at most $\epsilon$, cf. \eqref{eq:barrier_is_low_inside}.  Recall $\epsilon < \n/2$. We also bounded the last summand by using that $\z[0], \xrast \in \B$ so $\norm{\z[0]-\xrast}_2^2 \leq \n\omega^2$.

    At this point, the only free parameters left are $\Ck[0]$ (via $\etak[0]$) and $T$. We set $\etak[0] = \frac{1}{3\L}$ so that $\Ck[0]=1$. And we have that $\Ck[T] = 3\L\etak[T] = 3\L\etak[0](1-\tau)^{-T} =  (1-\tau)^{-T}$. So if we pick $T$ such that
\begin{equation}\label{eq:choosing_T}
\frac{1}{\Ck[T]} = (1-\tau)^{T} \leq \frac{\epsilon}{4\n\log(2\m\n)},
\end{equation}
we will obtain $\circled{2}$. We will pick the smallest $T$ that satisfies \eqref{eq:choosing_T}. That is,
\[
    T = \left\lceil \frac{\log(4\n\log(2\m\n)/\epsilon)}{\log(1/(1-\tau))} \right\rceil \leq \left\lceil3\L\log\left(\frac{4\n\log(2\m\n)}{\epsilon}\right)\right\rceil = \bigotilde{\n/\epsilon}.
\]
    On the other hand, by definition of $T$ as the minimum natural number satisfying \eqref{eq:choosing_T}, we have,
\[
    (1-\tau)^T = \frac{\etak[0]}{\etak[T]} \geq \frac{\epsilon}{4\n\log(2\m\n)}(1-\tau).
\]
    We can use this inequality to show $\etak \leq \frac{1}{4}$, for all $k\in[T]$, which is used in the proof of \cref{lemma:gradient_descent_compensates_packing_regret}. It is enough that $\circled{1}$ below is satisfied:
    \begin{equation}\label{eq:bound_by_one_fourth}
        \etak \leq \etak[T] \leq \frac{4\n\log(2\m\n)\etak[0]}{\epsilon} \frac{1}{1-\tau} = \frac{4\n\log(2\m\n)}{3 \L\epsilon} \frac{3\L}{3\L-1} \circled{1}[\leq] \frac{1}{4}.
\end{equation}
If $\L \geq \frac{16\n\log(2\m\n)}{3\epsilon} + \frac{1}{3}$ then $\circled{1}$ holds, and we chose $\L$ to satisfy this inequality.

    We note we could increase $T$ by a factor $C>1$ inside of its $\log$ in the numerator, so that the error obtained in $\circled{2}$ is $\epsilon/C$. However, the requirement on $\L$ above would increase by a factor of $C$, so we would end up with an extra factor of $C$ in the value of $T$. Also, the reduction caused by the smoothing already incurs in an $\epsilon$ additive error, cf. \cref{prop:optimizing_the_smoothened_function_is_enough}. Part $1$ of \cref{prop:optimizing_the_smoothened_function_is_enough} could use $x =\log(1-\frac{\epsilon}{\n C})\ones_{\n} +\xasthat$ so that inequality \eqref{eq:reduction_to_smooth_part_one} ends up being bounded by $O(\epsilon/C)$, but that would require to have $\beta$ be $C$ times smaller for $\eqref{eq:barrier_is_low_inside}$ to work. This would also require to make $\L$ larger by a factor of $C$ so we would also end up having the $C$ in the total number of iterations to obtain an $O(\epsilon)$-optimizer.

    In conclusion, we obtain an $\epsilon$ minimizer of $\fr$ in $\bigotilde{\n/\epsilon}$ iterations and by \cref{prop:optimizing_the_smoothened_function_is_enough}, we get that $\bar{x}$ is a feasible point that is a $5\epsilon$ optimizer of Problem~\eqref{eq:primal_problem}. Finally, we note that each iteration of the algorithm can be implemented in $O(\N)$ operations, that are distributed, where the bottleneck is the computation of the gradient. From the definition of the gradient, it is clear that it can be computed in our distributed model of computation and that each agent only needs their local variables for the rest of the steps in the algorithm.
\end{proof}

}

{

\renewcommand\D{\newlink{def:D}{\mathcal{D}}}
\newcommand\Dp{\newlink{def:D_plus}{\mathcal{D}^+}}
\renewcommand\P{\newlink{def:P}{\mathcal{P}}}

\newcommand\concat{\newlink{def:concat}{\operatorname{concat}}}

\newcommand\ghat{\newlink{def:dual_proxy}{\hat{g}}}
\renewcommand\c{\newlink{def:centroid}{c}}
\newcommand\tildem{\newlink{def:dimension_of_MW}{\widetilde{m}}}

\newcommand\lambdaast{\newlink{def:dual_optimizer}{\lambda^{\ast}}}

\newcommandx*\It[1][1=t, usedefault]{\newlink{def:I_t}{I_{#1}}}
\newcommand\Igeneric{\newlink{def:I_generic}{I}}

\newcommandx*\AI[2][1=\It, 2= , usedefault]{\newlink{def:A_It}{A_{#1}^{#2}}}

\newcommandx*\epsT[2][1=t, 2= , usedefault]{\newlink{def:epsilon_t}{\oldepsilon_{#1}^{#2}}}

\newcommand\tausub[1]{\newlink{def:width_parameter_tau_sub_delta}{\tau_{#1}}}
\newcommand\sigmasub[1]{\newlink{def:width_parameter_sigma_sub_delta}{\sigma_{#1}}}
\newcommand\tautrue{\newlink{def:width_parameter_true_tau}{\tau}}
\newcommand\sigmatrue{\newlink{def:width_parameter_true_sigma}{\sigma}}

\newcommandx*\lossk[2][1=k, 2= , usedefault]{\newlink{def:losses_dual}{\ell^{(#1)}_{#2}}}

\newcommandx*\h[2][1=\lambda, 2= , usedefault]{\newlink{def:constraint_h}{h^{#1}_{#2}}}
\renewcommandx*\p[2][1=\lambda, 2= , usedefault]{\hyperlink{def:centroid_point_p}{\color{black}p^{#1}_{#2}}}

\newcommandx*\etat[1][1= , usedefault]{\newlink{def:learning_rate_MW}{\eta_{#1}}}

\newcommandx*\Kt[1][1= , usedefault]{\newlink{def:K_number_of_iterations_of_PST_alg}{K_{#1}}}

\newcommandx*\Lambdak[2][1=k, 2= , usedefault]{{\protect\hyperlink{def:weights_of_MW}{\color{black}\Lambda^{(#1)}_{#2}}}}

\newcommandx*\signwidth[2][1=\sigma, 2=\tau, usedefault]{{\protect\hyperlink{def:sign_indicator_for_max_of_width_parameters}{\color{black}\operatorname{sign}(#1, #2)}}}

\newcommandx*\lens[2][1= , 2=\omega\delta, usedefault]{\newlink{def:lens_of_a_point}{\mathcal{L}_{#2}(#1)}}

\renewcommand\v{\newlink{def:aux_point_v_for_oracle}{v}}

\newcommandx*\lambdas[1][1= , usedefault]{\newlink{def:lambda_current_solution}{\lambda^{(s)}_{#1}}}
\newcommandx*\s[1][1= , usedefault]{\newlink{def:current_solution_constraint}{s_{#1}}}
\newcommandx*\ps[1][1= , usedefault]{\c(\s)_{#1}}
\newcommandx*\lambdaq[1][1= , usedefault]{\newlink{def:lambda_query_oracle}{\lambda^{(q)}_{#1}}}
\newcommandx*\q[1][1= , usedefault]{\newlink{def:query_oracle_constraint}{q_{#1}}}
\newcommandx*\pq[1][1= , usedefault]{\c(\q)_{#1}}
\newcommandx*\lambdao[1][1= , usedefault]{\newlink{def:lambda_output_oracle}{\lambda^{(o)}_{#1}}}
\newcommandx*\ho[1][1= , usedefault]{\c^{-1}(o)_{#1}}
\renewcommandx*\o[1][1= , usedefault]{\newlink{def:output_of_oracle}{o_{#1}}}

\section{Optimizing the dual problem} \label{sec:dual}
In this section, we reduce the dual problem~\eqref{eq:dual_problem} to the feasibility problem of a packing \LP{}, which we call the \emph{proxy} problem. Then, we show how we can use the \PST{} framework \citep{plotkin1995fast} to approximately solving the proxy. Our algorithm relies on a carefully built geometric oracle whose \textit{width parameters} can be guaranteed to decrease with the optimality gap. We perform a sequence of restarts, and after each of them we can guarantee lower oracle width. This allows for reducing the overall complexity.

\subsection[The dual 1-fair packing problem as a feasibility packing LP]{The dual 1-fair packing problem as a feasibility packing {\protect\LP{}}}

Let $\newtarget{def:P}{\P} \defi \{x\in\Rp^{\n} : Ax \leq \ones_{\m} \}$ be the feasible region of the primal problem~\eqref{eq:primal_problem}. In this section, we identify (non-negative) covering constraints of the form $\innp{h, x} \leq 1$ with vector $h\in \Rp^{\n}$. Recall we assume without loss of generality that $A$ satisfies \eqref{eq:normalization_of_A}, that is, the maximum entry of each column is $1$. This implies that $\ecanonical[i]\in \P$ for all $i \in [\n]$, and $\P\subseteq [0,1]^{\n}$. For this reason, we also assume in this section and without loss of generality that $A$ contains the constraints $\{\ecanonical\}_{i=1}^{\n}$, i.e., $x_i\leq 1$, for $i\in[\n]$. For convenience, we assume they are the first $\n$ rows of $A$.

Let $\newtarget{def:dual_optimizer}{\lambdaast}\in\simplex{\m}$ be an optimal solution of Problem~\eqref{eq:dual_problem} and let $\h[\lambdaast]\defi A^T\lambdaast\in \Rp^{\n}$.

By strong duality of the Lagrange dual, we can reconstruct the optimal primal solution as $\xast = (1/(\n\h[\lambdaast][1]), \dots, 1/(\n\h[\lambdaast][\n]))$.

This motivates the definition of the \emph{centroid map}:
\[
    \newtarget{def:centroid}{\c}(h) \defi \left(\frac{1}{\n h_1}, \dots, \frac{1}{\n h_{\n}}\right);\ \ \ \ \c^{-1}(x) \defi \left(\frac{1}{\n x_1}, \dots, \frac{1}{\n x_{\n}}\right),
\]
where $h,x\in\Rp^{\n}$  are constraints and points, respectively. Despite the two maps above being the same function we distinguish between $\c$ and $\c^{-1}$ to unambiguously refer to constraints or points.  The name of the centroid map is motivated by the fact that, for any constraint $h \in \Rp^{\n}$, the point $\c(h)$ is the geometric centroid of the simplex $\{x\in\Rp^{\n} : \langle h, x\rangle = 1\}$. Given a $\lambda \in \simplex{\m}$, we define its associated constraint $\newtarget{def:constraint_h}{\h[\lambda]}\defi A^T\lambda$. In addition, we define the centroid $\newtarget{def:centroid_point_p}{\p[\lambda]} \defi \c(A^{T}\lambda) = \c(\h[\lambda])$. Note that we can efficiently compute $\h[\lambda], \p[\lambda]$ from $\lambda$, and $\h[\lambda]$ from $\p[\lambda]$ and viceversa. However, we cannot obtain the coefficients $\lambda$ efficiently from $\h[\lambda]$ or $\p[\lambda]$, as this amounts to solving a linear program.

An important property is that $\h[\lambdaast]$ is the unique constraint of the form $\h[\lambda]$ such that $\c(\h[\lambda])\in \P$, for some $\lambda\in\simplex{\m}$. It is the optimizer because if $\h[\lambdaast]$ has $\c(\h[\lambdaast])\in \P$, then $(\p[\lambdaast], \lambdaast)$ satisfies the optimality conditions. It is unique because of strong convexity of $-\log(\cdot)$; any other dual optimizer $\bar{\lambda}^\ast$ will have $A^T\lambdaast = A^T\bar{\lambda}^{\ast}$. The following lemma shows an approximate version of this property.

\begin{lemma}\label{lemma:the_proxy_is_indeed_a_proxy}
  Let $\h[\lambda]$ for $\lambda\in\simplex{\m}$, such that $A \p[\lambda] \leq (1+\epsilon)\ones_{\m}$. Let $\lambdaast$ be the minimizer of Problem~\eqref{eq:dual_problem}, of objective function $\g$. Then, $\g(\lambda)-\g(\lambdaast) \leq \n \log (1+\epsilon) \leq \n\epsilon$.
\end{lemma}

\begin{proof}
	As $A \p[\lambda]  \leq (1+\epsilon)\ones_{\m}$ we have $A\frac{\p[\lambda]}{1+\epsilon} \leq \ones_{\m}$ and hence $\frac{\p[\lambda] }{1+\epsilon}$ is primal feasible. Therefore,
\begin{align}
   \begin{aligned}
        \g(\lambda)-\g(\lambda^*) & \leq \g(\lambda) - \f\left(\frac{\p[\lambda]}{1+\epsilon}\right) \\
        &= -\sum_{i\in [\n]} \log (A^T\lambda)_i - \n\log \n - \sum_{i\in[\n]} \log \left(\frac{1}{\n(1+\epsilon)(A^T\lambda)_i}\right)\\
	& = \n \log(1+\epsilon) \leq \n\epsilon.
   \end{aligned}
\end{align}
\end{proof}

This lemma motivates the minimization of $\max_{i\in[\m]} \langle A_i, \c(A^T\lambda)\rangle$ for $\lambda \in \simplex{\m}$ as a proxy for solving Problem~\eqref{eq:dual_problem}. Furthermore, if we optimize over the set $\{p^{\lambda}=\c(A^T\lambda) : \lambda \in \simplex{\m}\}$, we end up with a feasibility problem in a packing \LP{}. However, this set is not convex in general, which is a requirement of the \PST{} framework we intend to use. Fortunately, we can optimize over a larger and convex set while preserving the minimizer. Indeed, define the sets of constraints $\newtarget{def:D}{\D}\defi \{A^T \lambda: \lambda \in \simplex{\m}\} = \conv(\{A_1,\dots,A_{\m}\})$ and $\newtarget{def:D_plus}{\Dp}\defi \{ A^T \lambda - \mu\geq \zeros: \lambda \in \simplex{\m}, \mu\in\Rp^{\n} \} = \{h \in \Rp^{\n} : h \leq \h[\lambda], \text{ for } \h[\lambda] \in \D\}$. For a constraint $h \in \Rp$, we have by polyhedral duality that $h\in \Dp$ if and only if $\innp{ h, x} \leq 1$ for all $x\in \P$. 
In other words, $\Dp$ is exactly the set of positive constraints $\innp{h,x}\leq 1$ with $h\in \Rp^{\n}$ that are feasible for $\P$. We can think about the optimization of Problem~\eqref{eq:dual_problem} as
\[
    \min_{\h[\lambda]\in \D}\{-\sum_{i=1}^{\n}\log(\h[\lambda][i])-n\log(\n)\} \circled{1}[=] \min_{h\in \Dp}\{-\sum_{i=1}^{\n}\log(h_i)-\n\log(\n)\},
\]
where $\circled{1}$ comes from the following observation: since $h\in\Dp$ if and only if there is $\h[\lambda]\in \D$ with $h\leq \h[\lambda]$, and the expression $-\sum_{i=1}^{\n}\log(h_i)-\n\log(\n)$ is strictly decreasing in every $h_i$ then no $h\in \Dp \setminus \D$ could ever minimize the right hand side.
Consequently, the minimizer of both problems is the same. This results in the following proxy, motivated by \cref{lemma:the_proxy_is_indeed_a_proxy}:
\begin{align}\label{eq:dual_problem_proxy}\tag{1FPD-Proxy}
    \min_{p \in \c(\Dp)} \Big\{\newtarget{def:dual_proxy}{\ghat}(p)\defi \max_{i\in[\m]} \langle A_{i}, p \rangle \Big\}.
\end{align}

By \cref{lemma:the_proxy_is_indeed_a_proxy}, the optimizer of this problem must be $\p[\lambdaast]$.  Note $\ghat(\p[\lambdaast]) = 1$ so we want a $p$ such that $\ghat(p) \leq 1+\epsilon/n$. We solve this as an approximate feasibility problem in a packing \LP{} by using the \PST{} framework over $\c(\Dp)$, which is convex, according to the following lemma.
\begin{lemma}\label{lemma:D+_convex}
  Let $p^{(1)}, \dots, p^{(k)}$ be points in $\c(\D)$, and let $\zeta\in \simplex{k}$ be coefficients. Then, we have ${\c(\zeta_1 \c^{-1}(p^{(1)})+ \dots + \zeta_k \c^{-1}(p^{(k)})) \leq \zeta_1 p^{(1)} + \dots +\zeta_k p^{(k)}}$.  Consequently, $\c(\Dp)$ is convex.
\end{lemma}

\begin{proof}
    Let us look at one coordinate $i\in [\n]$. By the weighted harmonic-arithmetic inequality:
\begin{align}
   \begin{aligned}
       \c(\zeta_1 \c(p^{(1)})+ \dots + \zeta_k \c(p^{(k)}))_i &=\left(\zeta_1 \frac{1}{p_{i}^{(1)}} + \dots +\zeta_k \frac{1}{p_{i}^{(k)}}\right)^{-1} \\
        &\leq \zeta_1 p_{i}^{(1)} + \dots+ \zeta_k p_{i}^{(k)} =\left(\zeta_1 p^{(1)} + \dots +\zeta_k p^{(k)}\right)_i.
   \end{aligned}
\end{align}
	Now we prove convexity of $\c(\Dp)$. Let $q^{(1)},\dots, q^{(k)}\in \c(\Dp)$. Since every constraint in $\Dp$ is coordinate-wise smaller than some constraint in $\D$, it follows that every point $q^{(j)} \in \c(\Dp)$ is coordinate-wise larger than some point $p^{(j)} \in \c(\D)$, that is $q^{(j)}\geq p^{(j)}$ for $j\in[k]$. Thus we obtain
	\[
        \sum_{j\in[k]} \zeta_j q^{(j)} \geq \sum_{j\in[k]} \zeta_j p^{(j)} \circled{1}[\geq] \c\left(\sum_{j\in[k]} \zeta_j \c(p^{(j)}) \right) \circled{2}[\in] \c(\D).
    \]
    Inequality $\circled{1}$ is the first part of the lemma above, and the membership $\circled{2}$ is due to the convexity of $\D$, which is a polytope. The constraints in $\c(\Dp)$ are exactly the ones that are coordinate-wise greater than some constraint in $\c(\D)$, so we have $\sum_{j\in[k]} \zeta_j q_j \in \c(\Dp)$.
\end{proof}

\begin{figure}
    \centering
    
    \includegraphics[width=0.9\linewidth]{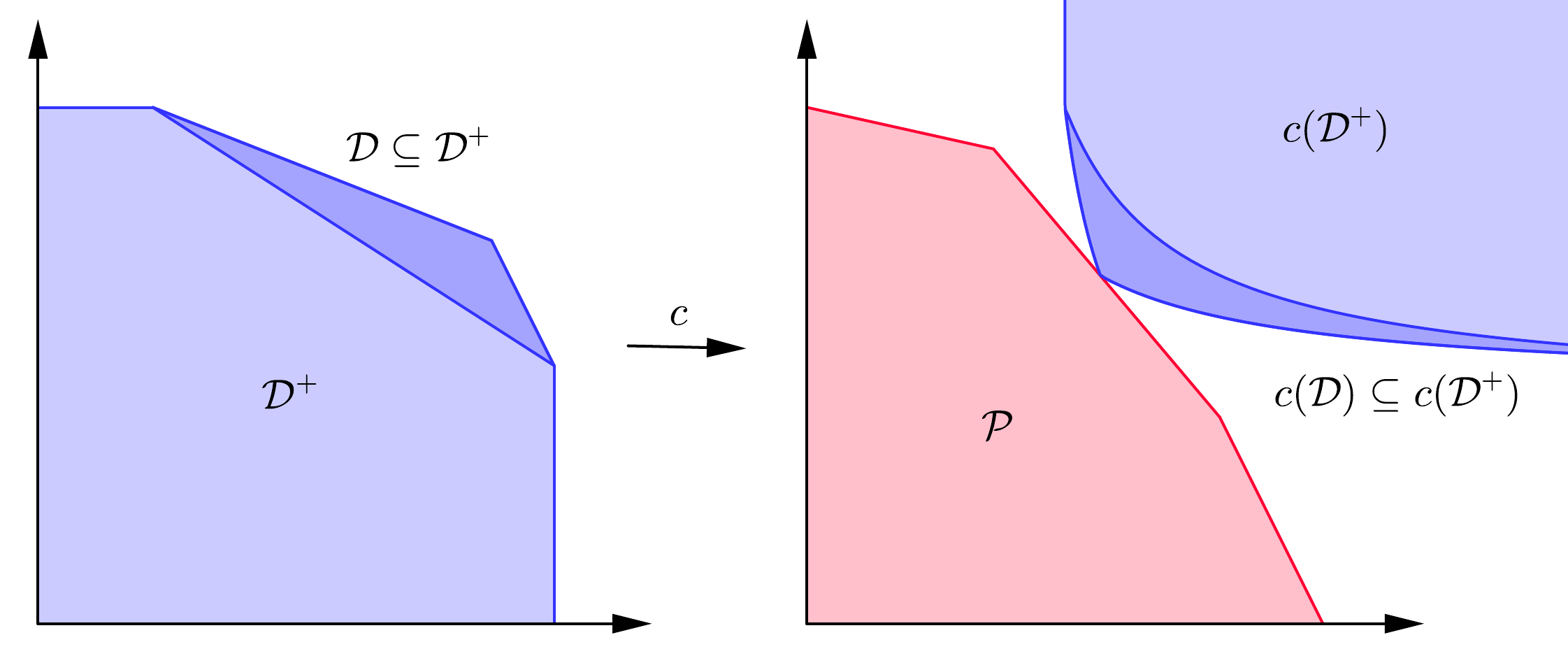}
    \caption{Left, the dual polytope $\Dp$ containing $\D$. The centroid $\c(\cdot)$ maps dual points (i.e., constraints) to primal points. Right, $\P$ and the image of $\D$ and $\Dp$ via $\c$. Note $\c(\Dp)\cap \P$ is exactly one point, which is also contained in $\c(\D)$. And note $\c(\D)$ can be non-convex, but $\c(\Dp)$ is convex.}
    \label{fig:primal_dual_plot}
\end{figure}

\subsection[The PST algorithm]{The {\protect\PST{}} algorithm}
\label{ssec:PST_algorithm}

Plotkin, Shmoys, and Tardos designed a framework (\PST{}) for solving \LP{} \citep{plotkin1995fast}. This result was improved by \citep{arora2012multiplicative} for the case of packing and covering \LP{}. We can use this framework to solve Problem~\eqref{eq:dual_problem_proxy} if we can provide a good feasibility oracle, as explained in the sequel. The \PST{} framework focuses on checking the feasibility of $Ap \leq \ones_{\m}$, with $p$ in a convex set $\X$.  Then, it obtains either a certificate of infeasibility of the problem or computes a point $p\in \X$ such that $Ap \leq (1+\epsilon)\ones_{\m}$, for some given $\epsilon > 0$. The \PST{} framework works by calling an oracle which solves the simpler feasibility problem of finding
\begin{equation}\label{eq:oracle_guarantee}
    p\in\X \text{ such that } \innp{A^T \Lambda, p} \leq 1,
\end{equation}
for some distribution $\Lambda \in \simplex{\m}$. That is, given a single constraint $\h[\Lambda]$, written as a convex combination of the constraints defined by $A$, the query asks for a point $p\in \X$ that satisfies the constraint. In our case, we apply the framework to $\X=\c(\Dp)$ to find a point $p\in \c(\Dp)$ with $Ap\leq (1+\epsilon)\ones_{\m}$. We can apply \PST{} because $\c(\Dp)$ is convex. Our oracle subproblems are always solvable because $\p[\lambdaast]\in \P$ satisfies them all. We use the following formulation of the \PST{} and multiplicative weights (\newtarget{def:acronym_multiplicative_weights}{\MW{}}) algorithms for packing \LP{}, which is a slight variation of \citep{arora2012multiplicative}. The \MW{} algorithm and its guarantees are presented in \cref{lemma:multiplicative_weights}.

\begin{lemma}[\PST{} guarantee]\linktoproof{lemma:PST_guarantee} \label{lemma:PST_guarantee}
    Let $\sigma, \tau \in\R_{>0}$. For a target accuracy $\epsilon \in (0, 4\min\{\sigma,\tau\}]$, run the \MW{} algorithm of \cref{lemma:multiplicative_weights} with $\delta = \epsilon/2$, losses $\lossk  \defi \ones_{\m} - A p^{(k)}$ assumed to be in $[-\sigma, \tau]^{\m}$, where $p^{(k)}$ is the point the oracle outputs at iteration $k$ when the constraint that is queried is given by $A^T (\Lambdak[k]/\norm{\Lambdak[k]}_1)$, and where $\Lambdak[k]$ are the weights computed by the \MW{} algorithm. Then, after $\newtarget{def:K_number_of_iterations_of_PST_alg}{\Kt[]}=\frac{32\sigma\tau\log(\m)}{\epsilon^2}$ iterations, we obtain a solution $\bar{p} \defi \frac{1}{\Kt[]}\sum_{k=1}^{\Kt[]} p^{(k)}$ that satisfies $\ghat(\bar{p})\leq 1+\epsilon$.
\end{lemma}

In order to give a solution of Problem~\eqref{eq:dual_problem}, we are also interested in recovering some $\lambda\in \simplex{m}$ such that $\ghat(\p[\lambda])$ is small. Assume the oracle returns a point $p^{(k)}=\p[\lambda^{(k)}]$ for some $\lambda^{(k)}$ it can also provide. Then, even if $\bar{p} \in \Dp \setminus \D$, we have by \cref{lemma:D+_convex} that $\bar{\lambda} \defi \frac{1}{K}\sum_{k=1}^K \lambda^{(k)}$ defines a point $\p[\bar{\lambda}] \in \D$ such that $\p[\bar{\lambda}] \leq \bar{p}$. Hence $\ghat(\p[\bar{\lambda}]) \leq \ghat(\bar{p}) \leq 1+\epsilon$, so $\bar{\lambda}$ can be used as our dual solution.

Thus our task is to construct an oracle with good enough $\sigma$ and $\tau$, which are called the width parameters of the oracle. We also need to make sure that $\epsilon \leq 4\min\{\sigma,\tau\}$, and that we can output a point $\p[\lambda]$ in $\D$, and a corresponding $\lambda$. Regardless, this algorithm runs in $\bigotilde{\sigma\tau/\epsilon^2}$ iterations, which is slower than what we aim for, for constant $\sigma$ and $\tau$. Our approach to obtain a fast algorithm is to design an adaptive feasibility oracle. We provide our best solution $\bar{\lambda}^{(t)}$ to the oracle, which satisfies $\ghat(\p[\bar{\lambda}^{(t)}]) \leq 1+\delta$, for some $\delta >0$. Given such a $\delta$-approximate solution we identify a set of indices $\newtarget{def:I_t}{\It}\subseteq [\n]$ such that the constraints $\{A_i, i\not\in \It\}$ are redundant, in the sense that the oracle is guaranteed to return points satisfying them even if they are removed from the problem. We remove these constraints since they would yield large values of $\tau$. Hence, we can run our \MW{} algorithm with the remaining $\abs{\It}$ constraints. Denote $\newtarget{def:A_It}{\AI}$ the matrix that has $\{A_i:i \in \It\}$ as rows, in increasing order of $i$. In this case if $\Lambda \in \simplex{\abs{\It}}$, we denote $\h[\Lambda]=\AI[\It][T]\Lambda$ and $\p[\Lambda]=\c(\AI[\It][T]\Lambda)$ accordingly. Our oracle, when given a query $\h[\Lambda]$, uses both constraints $\h[\Lambda]$ and $\h[\bar{\lambda}^{(t)}]$ to return a point $p$ satisfying the oracle condition~\eqref{eq:oracle_guarantee} and such that the loss $\ones_{\abs{\It}}-\AI p$ is in $[-\sigmasub{\delta}, \tausub{\delta}]^{\abs{\It}}$, with
\begin{equation}\label{eq:oracle_width}
    \newtarget{def:width_parameter_tau_sub_delta}{\sigmasub{\delta}} \defi 
        \begin{cases} 
            \sqrt{2\delta \n} + 2\delta \n & \text{ if } \delta \leq 2 \\
            \frac{1+2\delta}{1+\delta}\max_{i\in[\n]}\{1/\h[\bar{\lambda}^{(t)}][i]\}-1 & \text{ if }\delta > 2\\
        \end{cases}
        ,\quad \quad\quad \newtarget{def:width_parameter_sigma_sub_delta}{\tausub{\delta}} \defi \min\{3\sqrt{2\delta \n}, 1\}.
\end{equation}
In fact, $\sigmasub{\delta}$ can be defined as the minimum of its two expressions above, regardless of the value of $\delta$. But we shall use this definition for our algorithm. In the next section, we present the construction and analysis of such an oracle and the set $\It$. We observe that, because the parameters depend on $\delta$, we can restart the \MW{} algorithm, and run it in several stages, indexed by $t=0,\dots, T$. This allows for incrementally reducing the width parameters and obtaining a better overall complexity, as we prove in the following theorem.

\begin{algorithm}
    \caption{Optimization of the dual of 1-fair packing with oracle $\oracle$} 
    \label{alg:pst_mw}
\begin{algorithmic}[1]

    \REQUIRE Matrix $A \in \mathcal{M}_{\m\times \n}(\Rp)$ normalized as in \eqref{eq:normalization_of_A}. Accuracy $\epsilon\in (0, (\n-1)\n]$.
    \State $\bar{\lambda}^{(0)} \gets \concat(\ones_{\n}/\n, \mathbb{0}_{\m-\n})\in\simplex{\m}$; \ \ $\epsT[-1] = \n-1$ \Comment{$\p[\bar{\lambda}^{(t)}]$ is an $\epsT[t-1]$-minimizer of $\ghat$}
    \vspace{0.1cm}
    \hrule
    \vspace{0.1cm}
    \FOR {$t = 0 \text{ \textbf{to} } T \defi\max\{0, \lceil\log_2(2/(\epsilon/\n))\rceil\}$}
        \State $\It \gets \{i \in [\m] : \innp{A_i, \p[\bar{\lambda}^{(t)}]} \geq \frac{1+\epsT[t-1]}{1+2\epsT[t-1] \n +\sqrt{2\epsT[t-1] \n}}\}$ \Comment{Remove redundant constraints} \label{algo:line:filtering_constraints}
        \State $\textbf{ if } t\textbf{ is }0\textbf{ then }\epsT \gets \max\{2, \epsilon/\n\} \textbf{ else } \epsT \gets \epsT[t-1]/2$          \Comment{Next target accuracy}
        \State $\Lambdak[1] \gets \ones_{\abs{\It}}$ \Comment{Restart \MW{}}
        \FOR {$k = 1 \text{ \textbf{to} } \Kt[t]\defi 32\tausub{\epsT[t-1]} \sigmasub{\epsT[t-1]}\log(|\It|)/\epsT[t][2] =\bigotilde{\n/\epsT}$}
            \State $\lambda^{(k)}, \p[\lambda^{(k)}] \gets \oracle(\text{query }\Lambdak[k]/\norm{\Lambdak[k]}_1,\text{ current solution }\bar{\lambda}^{(t)},\text{ index set }\It)\in (\simplex{\m}, \c(\D))$
            \State $\lossk  \gets \ones_{\abs{\It}} - \AI \p[\lambda^{(k)}]$ \label{line:computing_losses}
            \State $\Lambdak[k+1] \gets \Lambdak[k]\odot(\ones_{\abs{\It}}- (\epsT/(4\tausub{\epsT[t-1]}\sigmasub{\epsT[t-1]}))\cdot\lossk )$ \Comment{\MW{} step}
        \ENDFOR
        \State $\bar{\lambda}^{(t+1)} \gets \frac{1}{\Kt[t]}\sum_{k=1}^{\Kt[t]} \lambda^{(k)}$\label{line:computing_average_of_lambdas} 
    \ENDFOR
    \State \textbf{return} $\bar{\lambda} \defi \bar{\lambda}^{(T+1)}$
    \ENSURE $A \p[\bar{\lambda}] \leq (1+\epsilon/\n)\ones_{\m}$, that is, $\ghat(\p[\bar{\lambda}]) \leq 1+\epsilon/\n$. Hence $\g(\bar{\lambda})-\g(\lambdaast)\leq \epsilon$.
\end{algorithmic}
\end{algorithm}

\begin{theorem}\label{thm:dual_guarantee}
    Let $\epsilon \in (0,\n(\n-1)]$. Suppose we have a feasibility oracle $\oracle$ satisfying \eqref{eq:oracle_guarantee} and \eqref{eq:oracle_width}, when we filter constraints according to \cref{alg:pst_mw}. Then, \cref{alg:pst_mw} computes, in $\bigotilde{\n^2/\epsilon}$ iterations, a point $\bar{\lambda}$ which is an $\epsilon$-minimizer of $\g$. Moreover, $\p[\bar{\lambda}]$ is an $(\epsilon/\n)$-minimizer of $\ghat$.
\end{theorem}

\begin{proof}
    Our aim is to use the \PST{} framework to obtain a $\widehat{\lambda} \in \simplex{\m}$ such that $\p[\widehat{\lambda}]$ is an $(\epsilon/\n)$-minimizer of $\ghat$ so that we can use \cref{lemma:the_proxy_is_indeed_a_proxy} to conclude $\widehat{\lambda}$ is an $\epsilon$-minimizer of $\g$.
    We will run the \MW{} algorithm several times, restarting the weights between executions. Let $\newtarget{def:epsilon_t}{\epsT} \defi 2^{1-t}$ for $t\geq 0$. At phase $t$, we aim to obtain a dual point $\bar{\lambda}^{(t+1)}$ such that $\p[\bar{\lambda}^{(t+1)}]$ is an $\epsT$-minimizer of $\ghat$. So we first aim for a $2$-minimizer and then we halve the accuracy sequentially. The initial point is $\bar{\lambda}^{(0)} \defi \newtarget{def:concat}{\concat}(\ones_{\n}/\n, \mathbb{0}_{\m-\n})$, where $\concat$ is defined as the concatenation of vectors, so that $\bar{\lambda}^{(0)}=(\frac{1}{n}, \dots, \frac{1}{n}, 0,\dots,0) \in \simplex{\m}$, with $n$ entries with value $\frac{1}{n}$. This point satisfies that $\h[\bar{\lambda}^{(0)}] = \ones_{\n}/\n$ and $\p[\bar{\lambda}^{(0)}] = \ones_{\n}$, since we assumed $x_i \leq 1$ are the first $\n$ constraints of $A$, i.e., the first $\n$ rows of $A$ are $\ecanonical$ for $i\in[\n]$. Because the maximum entry of $A$ is $1$, we have  $\ghat(\p[\bar{\lambda}^{(0)}]) \leq \n$, i.e., $\p[\bar{\lambda}^{(0)}]$ is an $(\n-1)$-minimizer of $\ghat$. Hence we denote $\epsT[-1] = n-1$ for convenience.

    So for $t=0$, the first phase, we seek to find a $2$-minimizer of $\ghat$. Thus, this is the only phase in which we use the second case of the value of $\sigmasub{\delta}$, cf. \eqref{eq:oracle_width}. We have $\tausub{\epsT[-1]} = 1$ and $\sigmasub{\epsT[-1]} = 2\n-2$ and $4\min\{\tausub{\epsT[-1]}, \sigmasub{\epsT[-1]}\} = 4$. Note that if $\epsilon/\n \geq 2$, we can actually stop earlier so $T$ in the algorithm is $0$, there is only one phase, and $\epsT[0]$ is actually set to $\epsilon/\n$. At this phase our \textit{good previous solution} is the initial point $\bar{\lambda}^{(0)}$. Assume first $\epsT[0] = 2$. The condition $2 = \epsT[0] \leq 4\min\{\tausub{\epsT[-1]}, \sigmasub{\epsT[-1]}\} =4$ is satisfied. Thus, according to \cref{lemma:PST_guarantee}, we reach the $\epsT[0]$-minimizer after $\Kt[0] = \bigotilde{\frac{\n}{\epsT[0][2]}} = \bigotilde{\n}$ iterations. If $\epsT[0]$ is $\epsilon/\n > 2$, we can artificially set $\tausub{\epsT[-1]}$ to a larger value, say $\epsilon/\n$, so that the condition $\epsT[0] \leq 4\min\{\tausub{\epsT[-1]},\sigmasub{\epsT[-1]}\}=4\tausub{\epsT[-1]}$ trivially holds, and the complexity is $\bigotilde{\frac{\epsilon/\n \cdot \sigmasub{\epsT[-1]}}{(\epsilon/\n)^2}} = \bigotilde{\frac{\n^2}{\epsilon}}$ iterations, which satisfies the theorem. In fact, for large $\epsilon$, i.e., for $\epsilon/\n>2$, just aiming for an $\widehat{\oldepsilon}=\exp(\epsilon/\n)-1$ is significantly faster and enough. The latter is true according to \cref{lemma:the_proxy_is_indeed_a_proxy} without bounding $\log(1+\widehat{\oldepsilon}) \leq \widehat{\oldepsilon}$.

    Now, if $T>0$, we run several phases of the \MW{} algorithm. Iteration $t>0$ takes $\Kt[t] \defi \frac{32\tausub{\epsT[t-1]} \sigmasub{\epsT[t-1]}\log(|\It|)}{\epsT[t][2]}=\bigotilde{\frac{\tausub{\epsT[t-1]}\sigmasub{\epsT[t-1]}}{\epsT[t][2]}}$ iterations by the \PST{} guarantee, cf \cref{lemma:PST_guarantee}. If $\epsT[t-1] > \frac{1}{\n}$ we have $\tausub{\epsT[t-1]}\cdot\sigmasub{\epsT[t-1]} = O(1 \cdot \epsT[t-1] \n)$ and if $\epsT[t-1] \leq \frac{1}{\n}$ we have $\tausub{\epsT[-1]}\cdot\sigmasub{\epsT[t-1]} = O(\sqrt{\epsT[t-1]\n} \cdot \sqrt{\epsT[t-1]\n})$. In any case, it is $\Kt[t] = \bigotilde{\frac{\n}{\epsT}}$, as $\epsT[t-1]/\epsT = 2 = O(1)$. The assumption $\epsT \leq 4\min\{\sigmasub{\epsT[t-1]}, \tausub{\epsT[t-1]}\}$ is satisfied in these phases. Indeed, if $\epsT\geq 1/\n$, we have $\epsT <\epsT[0] =2$ and $\sigmasub{\epsT[t-1]}\geq 1$, $\tausub{\epsT[t-1]}=1$ . In the case $\epsT<1/\n$, we have $4\min\{\tausub{\epsT[t-1]}, \sigmasub{\epsT[t-1]}\} \geq \sqrt{\epsT[t] \n}$ which is $>\epsT[t]$.

    If $T =\lceil\log_2(2/(\epsilon/\n))\rceil >0$, then $\epsT[T] \leq \epsilon/\n$ and the total number of iterations is
\[
    \sum_{t=0}^T \Kt[t] = \bigotilde{\sum_{t=0}^T \frac{\n}{\epsT}}=\bigotilde{\sum_{t=0}^T \n 2^{t-1}}=\bigotilde{\n 2^T} = \bigotilde{\frac{\n^2}{\epsilon}}.
\]

    The complexity of an iteration is $\bigotilde{\N}$, assuming $N\geq \m$ (recall $\m\geq \n$ since we added the constraints $x_i \leq 1$). Indeed, computing $\It$, \cref{line:computing_losses}, and computing the constraints of the query $\Lambdak[k]/\norm{\Lambdak[k]}_1$ and current solution $\bar{\lambda}^{(t)}$, to be used by the oracle, requires multiplying a vector by $A$ or a subset of its rows, which is $O(\N)$. The oracle query takes $\bigotilde{\n}$, and the rest of operations in \cref{alg:pst_mw} are simple and take $O(\m)$ time. Note the amortized complexity of \cref{line:computing_average_of_lambdas} is $O(\m)$ per iteration.
\end{proof}

\subsection[The PST oracle and the redundant constraints]{The {\protect\PST{}} oracle and the redundant constraints}
\label{ssec:PST_oracle}

In order to give a fast solution for Problem~\eqref{eq:dual_problem_proxy}, we want to design an adaptive oracle. That is, it should satisfy \eqref{eq:oracle_guarantee}, and it should output better points if it already knows a good approximate solution. Generically, assume that the oracle has access to a feasible constraint $\newtarget{def:current_solution_constraint}{\s}\defi A^T\lambdas \in \D$, with $\newtarget{def:lambda_current_solution}{\lambdas}\in\simplex{\m}$ satisfying $\p[\lambdas]=\ps$ is a point with $\ghat(\p[\lambdas]) \leq 1 + \delta$. In other words, $\ps$ is a $\delta$-approximate solution to Problem~\eqref{eq:dual_problem_proxy}. Let $\newtarget{def:query_oracle_constraint}{\q}\defi A^T \lambdaq \in \D$, with $\newtarget{def:lambda_query_oracle}{\lambdaq} \in \simplex{\m}$ be a query constraint. In \cref{alg:pst_mw}, $\lambdas$ and $\lambdaq$ correspond to $\bar{\lambda}^{(t)}$ and $\Lambdak[k]/\norm{\Lambdak[k]}_1$ respectively, where the latter is interpreted as having coefficients equal to zero for constraints $i\not\in \It$.

\begin{figure}[ht!]
    \centering
    \includegraphics[width=0.65\linewidth]{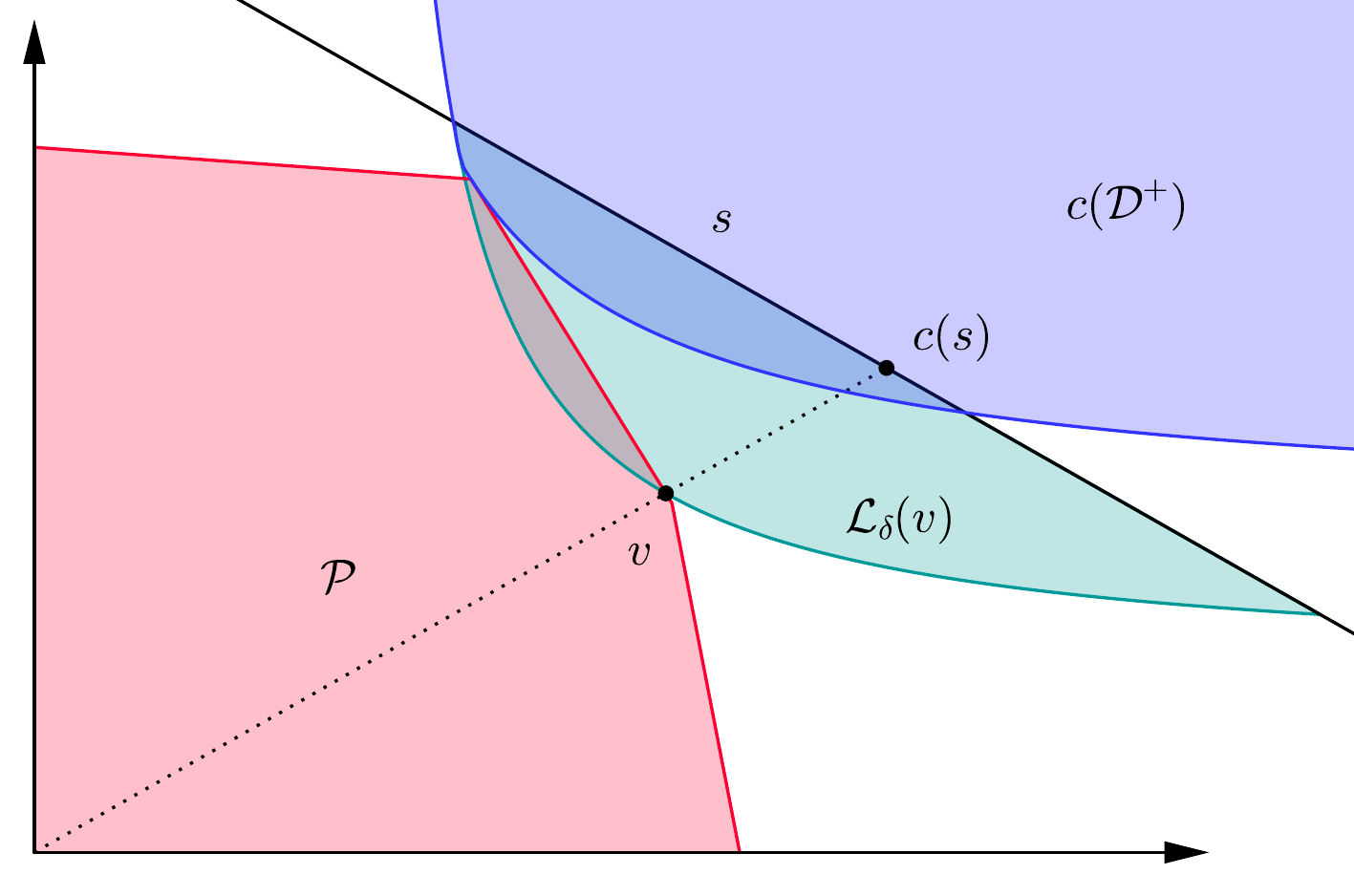}
    
    \caption{The lens ${\protect\lens[v][\delta]}$ given by a feasible solution $\s$, for $\omega=1$. In the actual algorithm, we have $\omega>1$, which defines a larger set in which the affine part is translated upwards.}
    \label{fig:geometric_intuition_of_the_lens}
\end{figure}

The main geometric idea is that by using the solution $s$, we can define a region whose size depends on $\delta$ containing the optimum $\p[\lambdaast]$ of Problem~\eqref{eq:dual_problem_proxy}. We will guarantee the oracle returns points in this region, which in turn means that the oracle returns points close to the optimum. We call this geometric object the \emph{lens of a point}:
\[
\newtarget{def:lens_of_a_point}{\lens[v][\omega\delta]} \defi \{x\in\Rp^{\n} : \innp{\c^{-1}(x), v} \leq 1, \innp{\c^{-1}(v),x} \leq 1+\omega \delta\},
\] 
where $\omega \in (1,2]$ is a parameter; we chose $\omega=2$ in the algorithm. We depict the lens in \cref{fig:geometric_intuition_of_the_lens}. We apply this definition to the point $\newtarget{def:aux_point_v_for_oracle}{\v} \defi \c(\s)/(1+\delta)$, which is in $\P$ because $\ghat(\s)\leq 1+\delta$ implies that $\c(\s)/(1+\delta)\in \P$.

We show in \cref{lem:lens_output} that, by using the bisection method, we can efficiently find a convex combination $(1-\mu)\s + \mu \q$ that will result in a constraint, whose centroid is the point $\o$ we output. It satisfies the first oracle condition, and is in the lens. Also, we can recover $\lambdao$ as $(1-\mu)\lambdas + \mu \lambdaq$.

If the oracle can output a point $\o$ in the lens, we will have low width parameters for the constraints $A_i$ that do not cover the lens, as we show in \cref{prop:our_oracle_satisfies_the_nice_properties}. That is, for one such constraint $A_i$, the corresponding loss is $1-\innp{A_i, \o} \in [-\sigmatrue, \tautrue]$. The other constraints could be problematic in terms of width parameters because $\innp{A_i, \o}$ could be much smaller than $1$, forcing $\tautrue$ to be large. However, we do not need to optimize over these constraints because $\ghat \geq 1$, so these constraints do not contribute to the $\max$ in its definition. This leads to the set of non-redundant constraints $\Igeneric$, which for efficiency reasons, we relax to those indices of constraints $A_i$ that do not cover a box that contains the lens, cf. \cref{lem:lens_box}. We prove this is good enough in terms of the width parameters. The computation of $\Igeneric$, which is done before each restart phase, requires $O(N)$ operations and each query to the oracle takes $O(\n\log(\frac{\n}{(\omega-1)\delta} +\frac{\n}{\omega-1}))$ operations, cf. \cref{lem:lens_output}.

\begin{algorithm}
    \renewcommand{\algorithmicensure}{\textbf{Output:}}
    \caption{Feasibility oracle $\oracle$} 
    \label{alg:oracle_3}
\begin{algorithmic}[1]
    \REQUIRE An approximate solution $\s \defi A^T \lambdas$, $\lambdas\in \simplex{\m}$, $\ghat(\ps) \leq 1+\delta$. Query constraint $\q=A^T \lambdaq$, $\lambdaq \in \simplex{\m}$. Precision parameter $1<\omega \leq 2$, default value $\omega=2$. Index set of \emph{non-redundant constraints} $\Igeneric = \{i \in [\m] : \innp{A_i, \ps} \geq \frac{1+\delta}{1+\omega\delta \n +\sqrt{\omega\delta \n}}\}$.

    \ENSURE $\newtarget{def:lambda_output_oracle}{\lambdao}\in \simplex{\m}$ and point $\newtarget{def:output_of_oracle}{\o}\defi\c(A^T\lambdao)\in\c(\D)$ such that
    \begin{enumerate}
        \item $\o$ is covered by $\q$, i.e., $\innp{\q, \o} \leq 1$. 
        \item If $i \in \Igeneric$ then $\innp{A_i,\o} \in [1-\tautrue, 1+\sigmatrue]$ where $\sigmatrue = \min (\sqrt{\omega\delta \n} +\omega\delta \n, \frac{1+\omega\delta}{1+\delta} \max_{i\in [\n]} \s_i^{-1} -1)$, $\tautrue = \min(3\sqrt{\omega\delta \n},1)$. 
        \item It satisfies all redundant constraints, i.e., $\innp{A_i, \o} \leq 1$, if $i \in [\m]\setminus \Igeneric$.
    \end{enumerate}

    \vspace{0.1cm}
    \hrule
    \vspace{0.1cm}
    \IF{$\innp{\s, \c(\q)} \leq \frac{1+\omega\delta}{1+\delta}$} \textbf{return} $\lambdao = \lambdaq$, \ \  $\o=\c(\q)$
    \ELSIF{$\innp{\q, \ps} \leq 1$} \textbf{return} $\lambdao = \lambdas$, \ \  $\o=\ps$
    \ELSE\  Find $\mu \in(0,1)$ s.t. $\innp{\s, \c((1-\mu)\s+ \mu \q)} \in (1, \frac{1+\omega\delta}{1+\delta})$ via the bisection method
    \ENDIF
    
      \State \textbf{return} $\lambdao = (1-\mu)\lambdas + \mu \lambdaq$,\ \  $\o=\c((1-\mu)\s + \mu \q)$
\end{algorithmic}
\end{algorithm}

}

{

\subsection{An application to the Yamnitsky-Levin algorithm via \ref{eq:dual_problem}} \label{sec:yl}

The Yamnitsky-Levin algorithm (\newtarget{def:acronym_yamnitsky_levin}{\YL{}}), or method of simplices \citep{yamnitsky1982}, belongs to a family of \LP{} algorithms that bound the feasible region by a geometric shape that is progressively decreased in volume, in order to solve the feasibility problem, using which one can solve \LP{}. \citep{levin1965algorithm} proposed the method of centralized splitting as a way to solve the linear feasibility problem. It starts with a simplex bounding the feasible region $P$ that is sequentially intersected with several halfspaces containing $P$. Then, it re-encloses the remaining polyhedron by a simplex that is smaller in volume. This algorithm later inspired the ellipsoid algorithm \citep{yudin1976informational}, that only needed the intersection with one halfspace. Later \citep{yamnitsky1982} proved that their method of simplices, which is a variant of the algorithm in \citep{levin1965algorithm}, can use only one intersection per iteration, and solves \LP{} in weakly polynomial time. In particular, the \YL{} algorithm computes the centroid $c^{(k)}$ of the current bounding simplex $\Delta^{(k)}$ at each iteration, and chooses an arbitrary constraint $h^{(k)}$ of the \LP{} problem that does not cover $c^{(k)}$. The simplex $\Delta^{(k+1)}$ results from substituting the facet that opposes the vertex $v^{(k)}$ that is farthest from $h^{(k)}$, by a convex combination of this facet and $h^{(k)}$. 

We note that since the \YL{} algorithm chooses constraints arbitrarily, one can first consider the set of constraints $\mathcal{H}_v$ whose normals are in the normal fan of any particular fixed vertex $v$, and one would optimize the volume of the simplex when the corner given by $v$ is fixed. The \YL{} algorithm would repeatedly run this procedure in stages, by pivoting to other vertices, until a feasible point is found or until the volume of the bounding simplex is small enough. In this case, a stage is equivalent to \eqref{eq:dual_problem}. One stage stops when all $h\in\mathcal{H}_v$ cover $c^{(k)}$, and this stopping condition is equivalent to reaching a $(1/\n)$-minimizer of~\eqref{eq:dual_problem_proxy}, for a matrix $A$ depending on $\mathcal{H}_v$. 

The classical analysis of the \YL{} algorithm guarantees that a stage finishes after $\bigotilde{\N \n^3}$ operations if we start with a natural non-informed initial simplex. Our \cref{alg:pst_mw} achieves this accuracy after $\bigotilde{\N \n^2}$ operations. See \cref{app:section:yl} for the description and analysis of the \YL{} stage.

}

\acks{
We thank Elias Koutsoupias for some suggestions and discussions that started the line of research of this work. David Martínez-Rubio and Francisco Criado were partially funded by the DFG Cluster of Excellence MATH+ (EXC-2046/1, project id 390685689) funded by the Deutsche Forschungsgemeinschaft (DFG). David Martínez-Rubio was also partially supported by EP/N509711/1 from the EPSRC MPLS division, grant No 2053152.
}

\clearpage

\printbibliography[heading=bibintoc] 

\clearpage
\appendix

\section[Missing proofs from Section \ref{sec:intro}]{Missing proofs from \cref{sec:intro}}

\begin{lemma} \label{lem:lagrange}
  Recall the definition of the primal problem \eqref{eq:primal_problem}:
\begin{align}
    \tag{1FP} 
    \max_{x\in \Rp^{\n}} \left\{ \f(x) \defi \sum_{i=1}^{\n} \log x_i : Ax \leq \ones_{\m} \right\}.
\end{align}

Then, its Lagrange dual can be formulated as:

\begin{align}
	\tag{1FP-Dual}
    \min_{\lambda \in \simplex{\m}} \left\{ \g(\lambda) \defi -\sum_{i=1}^{\n} \log (A^T \lambda)_i - \n \log \n  \right\}.
\end{align}

Where $\simplex{\m}\defi \{\lambda \in \R^{\m} : \sum \lambda_i = 1, \lambda \geq 0 \}$ is the $\m$-dimensional (probability) simplex.

\end{lemma}

\begin{proof} By definition of Lagrangian duality, the dual of \eqref{eq:primal_problem} is

  \begin{align*} \begin{aligned}
    \min_{y\geq 0} \left\{\max_{x\geq 0} \left\{ \sum_{i=1}^n \log(x_i) - y^T(Ax-\ones_m)\right\}\right\} =  \\
    \min_{y\geq 0} \left\{ \innp{y, \ones_m} + \max_{x\geq 0} \left\{ \sum_{i=1}^n \left( \log (x_i) - (y^T A)_i x_i \right)\right\}\right\}
  \end{aligned} \end{align*}

  We can explicitly compute the $x_i$ variables by differentiation, to obtain $x_i= (y^TA)_i^{-1}$. Thus the problem is equivalent to:

  \begin{align*}\begin{aligned}
    \min_{y\geq 0} \left\{\innp{y,\ones_m} + \sum_{i=1}^n\left(-\log(y^TA)_i -1\right)\right\}.
  \end{aligned}\end{align*}

  Now, let $\lambda = \frac{1}{\innp{y,\ones_m}} y$ and $t= \innp{y,\ones_m}$. Then, $y=t \lambda$ with $\lambda\in\simplex{\m}$ and we write the problem as:

  \begin{align*}\begin{aligned}
    \min_{\lambda\in\simplex{\m}} \min_{t\geq 0} \left\{t - \sum_{i=1}^n \log(\lambda^TA)_i -n\log t -n \right\}.
  \end{aligned}\end{align*}

  The value of $t$ minimizing $t-n\log t$ is $t=n$ by differentiation so the problem reduces to

\begin{align}
    \min_{\lambda \in \simplex{\m}} \left\{ -\sum_{i=1}^{\n} \log (A^T \lambda)_i - \n \log \n  \right\}.
\end{align}

\end{proof}

\section[Missing proofs from Section \ref{sec:primal}]{Missing proofs from \cref{sec:primal}}\label{app:proofs_of_primal_section}
{


\let\oldbeta\beta
\renewcommand\beta{\newlink{def:beta}{\oldbeta}}
\let\oldomega\omega
\renewcommand\omega{\newlink{def:omega}{\oldomega}}
\let\oldtau\tau
\renewcommand\tau{\newlink{def:tau}{\oldtau}}

\newcommandx*\truncgrad[2][1=j, 2=\fr, usedefault]{\newlink{def:truncgrad}{\overline{\nabla_{#1} #2}}}

\newcommandx*\nuk[2][1=k, 2= , usedefault]{\newlink{def:remainder_trunc_coord_grad}{\nu}^{( #1 )}_{#2}}

\renewcommandx*\x[2][1=k, 2= , usedefault]{\newlink{def:xk_algorithm}{x^{( #1 )}_{#2}}}
\renewcommandx*\y[2][1=k, 2= , usedefault]{\newlink{def:yk_algorithm}{y^{( #1 )}_{#2}}}
\renewcommandx*\z[2][1=k, 2= , usedefault]{\newlink{def:zk_algorithm}{z^{( #1 )}_{#2}}}

\newcommandx*\Ck[2][1=k, 2= , usedefault]{\newlink{def:Ck}{C_{#1}^{#2} }}
\newcommandx*\etak[2][1=k, 2= ,  usedefault]{\newlink{def:eta_k}{\eta_{#1}^{#2} }}

\newcommandx\fr{\newlink{def:f_r}{f_r}}
\newcommandx\fhat{\newlink{def:f_hat}{\hat{f}}}

\renewcommand\L{\newlink{def:L}{L}}

\newcommand\B{\newlink{def:Box}{B}}

\newcommand\xrast{\newlink{def:xrast}{x_r^\ast}}
\newcommand\xasthat{\newlink{def:xasthat}{\hat{x}^\ast}}

\newcommandx*\lossk[2][1=k, 2= , usedefault]{\newlink{def:losses_primal}{\ell^{(#1)}_{#2}}}

\begin{proof}\textbf{of \cref{prop:optimizing_the_smoothened_function_is_enough}.} \linkofproof{prop:optimizing_the_smoothened_function_is_enough}
    Recall $\beta=\frac{\epsilon}{6n\log(2mn^2 /\epsilon)}$ and that $\log(\cdot)$ is the natural logarithm. We will prove the proposition in three steps:
    \begin{enumerate}[label=\arabic*)]
        \item $\fhat(\xasthat)- \fhat(\xrast) \leq \fhat(\xasthat)+ \fr(\xrast) \leq 3\epsilon$.
        \item  The point $x_r^{\epsilon}$ satisfies $A\exp(x_r^{\epsilon}) \leq (1+\epsilon/\n)\ones_{\m}$.
        \item The point $\hat{u} = x_r^{\epsilon}-\log(1+\epsilon/\n)\ones_{\n}$ satisfies $\f(\xast)-\f(u) \leq \f(\exp(\xasthat))-\f(\exp(\hat{u}))  = \fhat(\xasthat) - \fhat(\hat{u}) \leq 5\epsilon$ and $A\exp(\hat{u}) \leq \ones_{\m}$.
    \end{enumerate}

    For the first part, take the point $x=\log(1-\epsilon/\n)\ones_{\n} + \xasthat \in \B$. It satisfies $A\exp(x) \leq (1-\epsilon/\n)\ones_{\m}$, because $\xasthat$ is feasible. Thus
    \begin{equation} \label{eq:barrier_is_low_inside}
        \frac{\beta}{1+\beta}\sum_{i=1}^{\m} (A\exp(x))_i^{\frac{1+\beta}{\beta}} \leq \m(1-\epsilon/\n)^{1/\beta} \leq \m\left(\frac{\epsilon}{2\m\n^2}\right)^6 \leq \epsilon.
    \end{equation}
    We used $\frac{\beta}{1+\beta} \leq 1$, $\frac{1+\beta}{\beta} \geq \frac{1}{\beta}$, and $(1-\epsilon/\n)^{\frac{\n}{\epsilon}} \leq e^{-1}$. Consequently, we have
\begin{align}
        \begin{aligned}\label{eq:reduction_to_smooth_part_one}
           \fhat(\xasthat) - \fhat(\xrast) &\circled{1}[\leq] \fhat(\xasthat) + \fr(\xrast)
           \circled{2}[\leq] \fhat(\xasthat) + \fr(x) \\
            &= \innp{\ones_{\n}, \xasthat} + \left(-\innp{\ones_{\n}, \log(1-\epsilon/\n)\ones_{\n}+ \xasthat} + \frac{\beta}{1+\beta}\sum_{i=1}^{\m} (A\exp(x))_i^{\frac{1+\beta}{\beta}} \right)  \\
           &\circled{3}[\leq] \n\log(\frac{1}{1-\epsilon/\n}) + \epsilon
           \circled{4}[\leq] 3\epsilon.  
       \end{aligned}
    \end{align}
    Above, $\circled{1}$ is true by definition of $\fr$ being $-\fhat$ plus a non-negative regularizer. The point $x$ is in $\B$ and $\xrast = \argmin_{x\in \B}\{\fr(x)\}$ so we have $\circled{2}$. Inequality $\circled{3}$ uses \eqref{eq:barrier_is_low_inside} and $\circled{4}$ uses $\log(x) \leq x-1$ and $\epsilon/\n \leq 1/2$.

    For the second part, suppose for the moment that there is some $i$ such that $(A\exp(x_r^{\epsilon}))_i > 1+\epsilon/\n$. In that case
    \[
        (A\exp(x_r^{\epsilon}))_i^{\frac{1+\beta}{\beta}} \geq (1+\epsilon/\n)^{(2\n/\epsilon) \cdot 3\log(2\m\n^2/\epsilon)} \geq  \left( \frac{2\m\n^2}{\epsilon}\right)^3, 
    \]
    since $(1+\epsilon/\n)^{2\n/\epsilon} \geq e$ when $\epsilon/\n \leq 1/2$. We have $x_r^{\epsilon}\in\B$ so $\fr(x_r^{\epsilon}) \geq -\innp{\ones_{\n}, x_r^{\epsilon}} + \frac{\beta}{1+\beta}\left( \frac{2\m\n^2}{\epsilon}\right)^3 \geq \frac{\beta}{2}\left( \frac{2\m\n^2}{\epsilon}\right)^3$. On the other hand, it holds for the point $y = -\log(\m\n)\ones_{\n}$ that
    \begin{align}
       \begin{aligned}
           \fr(y) &=    \n\log(\m\n) + \frac{\beta}{1+\beta}\sum_{i=1}^{\m} (A\exp(y))_i^{\frac{1+\beta}{\beta}}  \\
                   &\circled{1}[\leq] \n\log(\m\n) + \m\left(1/\m\right)^{\frac{1+\beta}{\beta}}
                   \leq \n\log(\m\n) + 1 \\
                   &\circled{2}[<] \frac{\beta}{2}\left(\frac{2\m\n^2}{\epsilon}\right)^3 -\epsilon
                   < \fr(x_r^{\epsilon}) -\epsilon,
       \end{aligned}
    \end{align}
    contradicting the assumption $\fr(x_r^{\epsilon})-\fr(\xrast) \leq \epsilon$, as we would obtain $\epsilon < \fr(x_r^{\epsilon}) - \fr(y) \leq \fr(x_r^{\epsilon}) - \fr(\xrast)$, since $y\in \B$. So it must be $(A\exp(x_r^{\epsilon}))_i \leq 1+\epsilon/\n$. Inequality $\circled{1}$ uses that the maximum entry of $A$ is $1$, and $\frac{\beta}{1+\beta} \leq 1$. One can show $\circled{2}$ by proving the stronger inequality that results from substituting $\beta$ by $\epsilon/(6\n\cdot 2\m\n^2/\epsilon)$, which is a lower value. Computing derivatives in both sides shows that this inequality holds if it does for $\m = 1$ and $\epsilon=\frac{\n}{2}$, and the latter is easy to check. 

    For the third part, we have $A\exp(\hat{u}) = A \frac{\exp(x_r^{\epsilon})}{1+\epsilon/\n} \leq \ones_{\m}$. And finally, putting all together we obtain
    \begin{align}
       \begin{aligned}
           \fhat(\xasthat) - \fhat(\hat{u}) &= \fhat(\xasthat) - \fhat(x_r^{\epsilon}) + \n\log(1+\epsilon/\n) \\
           &\leq \fhat(\xasthat) + \fr(x_r^{\epsilon}) + \n\log(1+\epsilon/\n) \\
           &\leq \fhat(\xasthat) + \fr(\xrast) + \epsilon + \n\log(1+\epsilon/\n) \\
           &\leq 4\epsilon + \n\log(1+\epsilon/\n)
           \leq 5\epsilon.
       \end{aligned}
    \end{align}
\end{proof}

\begin{lemma} \label{lemma:lower_bound_on_coordinates_of_x_ast}
    Let $A$ satisfy the normalization in \eqref{eq:normalization_of_A}, and let $\xast$ be the optimizer of Problem~\eqref{eq:primal_problem}. Then $\xast[i] \geq 1/\n$, for all $i\in [\n]$.
\end{lemma}

We now present the proof of \cref{lemma:iterates_remain_in_B}, which guarantees the iterates of \cref{alg:1_fair_packing} remain in the box $\B$.

\begin{proof}\textbf{of \cref{lemma:iterates_remain_in_B}.} \linkofproof{lemma:iterates_remain_in_B}
    For all $k \geq 0$, we have $\z[k] \in \B$ by definition. If we have that $\y[k-1] \in \B$, then $\x[k]\in\B$ since $\x[k]$ is a convex combination of $\y[k-1]$ and $\z[k-1]$. So we only have to prove that for all $k\geq 0$, we have $\y[k] \in \B$. We prove by induction that, for $k\geq 1$, it holds that $\y[k]$ is a convex combination of $\{\z[i]\}_{i=0}^k$ and that the weight of $\z[k]$ in this convex combination is $\frac{1}{\etak \L}$. Firstly, we have $\y[1] = (1-\frac{1}{\etak[1]\L}) \z[0] + \frac{1}{\etak[1]\L}\z[1]$ (recall $\x[0]=\z[0]$). Now assuming our property holds up to $k-1$, use the definition of $\y[k]$ and $\x[k]$, to compute $\y[k] = \tau \z[k-1] + (1-\tau)\y[k-1] + \frac{1}{\etak \L}(\z[k]-\z[k-1])$. This is an affine combination of the $\z[i]$'s, by induction hypothesis. Moreover, the weights add up to $1 = \tau + (1-\tau) +\frac{1}{\etak \L}-\frac{1}{\etak \L}$, and the weight on $\z[k]$ is $\frac{1}{\etak \L}$. So we only have to prove the weight on $\z[k-1]$ is $\geq 0$ in order to show that we indeed have a convex combination and not just an affine one. By induction hypothesis, we know the weight on $\z[k-1]$ coming from $\y[k-1]$ is $\frac{1}{\etak[k-1]\L}$. Hence, the weight on $\z[k-1]$ is $\tau + (1-\tau)\frac{1}{\etak[k-1]\L}-\frac{1}{\etak \L} = \tau > 0$, where the equality uses the definition of $\etak$.
\end{proof}

\begin{proof}
    The normalization ensures that $\ecanonical$ are feasible points, for $i\in [\n]$. That is, $Ae_i \leq \ones_{m}$ because each $A_{ij} \leq 1$. Since $\xast$ is the maximizer of Problem~\eqref{eq:primal_problem}, by the first order optimality condition we have $\innp{\nabla \f(\xast), x - \xast} \leq 0$, for any feasible point $x$. Suppose there is a coordinate $i\in[\n]$ such that $\xast[i] < \frac{1}{\n}$. Then, $\innp{\nabla \f(\xast), e_i - \xast} = \frac{1}{\xast[i]} - \sum_{j=1}^{\n} \xast[j]/\xast[j] > 0$, which is a contradiction.
\end{proof}

\begin{lemma}[Mirror Descent Lemma] \label{lemma:mirror_descent_lemma}
    Let $\X \subseteq \R^{\n}$ be a closed convex set and let $\psi:\X \to \R$ be a $1$-strongly convex map with respect to $\norm{\cdot}$. Let $\norm{\cdot}_\ast$ be the dual norm to $\norm{\cdot}$ and let $\newtarget{def:losses_primal}{\lossk}  \in \R^{\n}$ be an arbitrary loss vector. Given $z^{(k-1)} \in \X$, let $z^{(k)} \defi \argmin_{z\in\X}\{\breg(z, z^{(k-1)}) + \eta\innp{\lossk , z}\}$. Then, for all $u \in \X$ we have
    \begin{enumerate}[label=\alph*)]
        \item $ \eta \innp{\lossk , z^{(k-1)}-u}  \leq \frac{\eta^2}{2}\norm{\lossk }^2_\ast + \breg(u, z^{(k-1)})- \breg(u, z^{(k)})$.
        \item $\eta \innp{\lossk , z^{(k-1)}-u}  \leq \eta\innp{\lossk , z^{(k-1)}-z^{(k)}} + \breg(u, z^{(k-1)})- \breg(u, z^{(k)})$.\label{lemma:enum:mirror_descent_non_standard_lemma}
    \end{enumerate}
\end{lemma}

\begin{proof}
    We note that, by definition, we have $\frac{\partial}{\partial x}\breg(x,y) = \nabla \psi(x)-\nabla \psi(y)$. The lemma is due to
    \begin{align*}
       \begin{aligned}
           \innp{\eta \lossk , \z[k-1]-u} &= \innp{\eta \lossk , \z[k-1]-\z} + \innp{\eta \lossk , \z-u} \\
           &\circled{1}[\leq] \innp{\eta \lossk , \z[k-1]-\z} -\innp{\nabla\psi(z^{(k)})-\nabla\psi(z^{(k-1)}), \z-u} \\
           &\circled{2}[=] \innp{\eta \lossk , \z[k-1]-\z} - \breg(z^{(k)}, z^{(k-1)}) + \breg(u, z^{(k-1)}) -\breg(u, z^{(k)}) \\
           &\circled{3}[\leq] \frac{\eta^2}{2}\norm{\lossk }^2_\ast + \breg(u, z^{(k-1)})- \breg(u, z^{(k)}).
       \end{aligned}
    \end{align*}
    Inequality $\circled{1}$ comes from the first-order optimality condition of the definition of $z^{(k)}$, that is, $\innp{\nabla\psi(z^{(k)})-\nabla\psi(z^{(k-1)}) + \eta\lossk , u-z^{(k)}} \geq 0$ for all $u\in \X$. $\circled{2}$ is the triangle equality of Bregman divergences, and can be easily checked by using the definition.

    If we drop the term $- \breg(z^{(k)}, z^{(k-1)}) $ after $\circled{2}$, we obtain part $b)$ of this lemma. $\circled{3}$ leads to part $a)$, which is the classical mirror descent lemma. It uses the bound $\breg(z^{(k)}, z^{(k-1)}) \geq \frac{1}{2}\norm{z^{(k)}- z^{(k-1)}}^2$, which holds due to the strong convexity of $\psi$. And then we applied the inequality $\innp{v, w} - \frac{1}{2}\norm{w}^2 \leq \frac{1}{2}\norm{v}_{\ast}^2$ for $v, w \in \R^{\n}$, that holds by Cauchy-Schwarz and $\norm{v}_\ast\cdot\norm{w} \leq  \frac{1}{2}\norm{v}_{\ast}^2 + \frac{1}{2}\norm{w}^2$.

\end{proof}

In the next lemma, we prove a crucial generalization of \citep[Lemma 3.1]{diakonikolas2020fair}, which allows for showing our descent \cref{lemma:descent_lemma_coord_descent}.

\begin{lemma}\label{lemma:fair_packing_descent_lemma}
    Let $c \in [-1,1]^{\n}$ and let $\Delta \in \R^{\n}$ be defined as $\Delta_{j}=-\frac{c_{j} \beta}{4(1+\beta)} \truncgrad[j](x)$, for $j\in[\n]$. Then
\[
    \fr(x+\Delta)-\fr(x) \leq \sum_{j=1}^{\n} (1-\frac{c_j}{2}) \Delta_j  \nabla_j \fr(x).
\]
\end{lemma}

\begin{proof}\label{proof_of_descent_lemma}
    By using a Taylor expansion, there is a $t \in [0,1]$ such that
    \begin{equation}\label{eq:taylor_on_log_fr}
        \fr(x+\Delta)-\fr(x) = \innp{\nabla \fr(x), \Delta} + \frac{1}{2} \Delta^\top \nabla^2 \fr(x+t\Delta) \Delta.
    \end{equation}
    The gradient and Hessian of $\fr$ are given by
    \begin{align} \label{eq:value_of_grad_hessian_log_fr}
     \begin{aligned}
         \nabla_j \fr(x) &= -1  + \sum_{i=1}^{\m} (A\exp(x))_i^{\frac{1}{\beta}} a_{ij} \exp(x_j),\\
         \nabla^2_{jk} \fr(x) &= \ones_{\{j=k\}} \sum_{i=1}^{\m} (A\exp(x))_i^{\frac{1}{\beta}} a_{ij} \exp(x_j) \\
                 & \quad + \frac{1}{\beta}\sum_{i=1}^{\m} (A\exp(x))_i^{\frac{1}{\beta}-1} a_{ij}\exp(x_j) a_{ik}\exp(x_k).
     \end{aligned}
    \end{align}
    In order to control how much the function changes, we will require 
    \[
    \frac{1}{2} \nabla^2_{jk} \fr(x) \leq \nabla^2_{jk} \fr(x+t\Delta) \leq 2 \nabla^2_{jk} \fr(x).
    \] 
    We can guarantee the inequality on the right if we guarantee that each summand in the expression above does not grow by more than a factor of $2$, and respectively the one on the left if each summand does not decrease by more than a factor of $2$. Let $\Delta_{\max} \defi \max_{i\in[\n]}\{\Delta_i\}$ and $\Delta_{\min}\defi \min_{i\in[\n]}\{\Delta_i\}$. It suffices to have $\exp(\Delta_{\max})^{\frac{1}{\beta}+1} \leq 2$ and $\exp(\Delta_{\min})^{\frac{1}{\beta}+1} \geq 1/2$. Hence, it suffices to have for all $j \in [n]$, the following:
    \begin{equation}\label{eq:condition_on_delta_to_control_hessian}
         -\frac{\ln 2}{1+\frac{1}{\beta}} \leq \Delta_j \leq \frac{\ln 2}{1+\frac{1}{\beta}}.
    \end{equation}
    In fact, we will use $\Delta_j =-\frac{c_j}{4} \cdot \frac{\beta}{1+\beta}\truncgrad(x)$ for all $j\in[\n]$, which satisfy the condition since $\abs{c_j\truncgrad(x)} \leq 1$. In such a case, by using the function $\newtarget{def:sign}{\sign}(x)=1$ if $x \geq 0$ and $-1$ otherwise, we have:
    \begin{align} \label{eq:bound_on_hessian}
     \begin{aligned}
         \frac{1}{2} &\Delta^\top \nabla^2 \fr(x+t\Delta) \Delta \circled{1}[\leq] \sum_{j=1}^{\n} \sum_{i=1}^{\m} (A\exp(x))_i^{\frac{1}{\beta}}\Delta_j^2 a_{ij} \exp(x_j) \\
         &\quad + \frac{1}{2\beta}\sum_{i=1}^{\m} (A\exp(x))_i^{\frac{1}{\beta}-1}\sum_{j=1}^{\n}\sum_{k=1}^{\n} \Delta_j\Delta_k a_{ij}a_{ik}\exp(x_j)\exp(x_k)2^{\sign(\Delta_j\Delta_k)} \\
         & \circled{2}[\leq] \sum_{j=1}^{\n} \sum_{i=1}^{\m} (A\exp(x))_i^{\frac{1}{\beta}}\Delta_j^2 a_{ij} \exp(x_j) \\
         &\quad + \frac{1}{2\beta}\sum_{i=1}^{\m} (A\exp(x))_i^{\frac{1}{\beta}-1} \sum_{j=1}^{\n} \Delta_j^2 a_{ij}\exp(x_j) \cdot \sqrt{\sum_{j=1}^{\n}\sum_{k=1}^{\n} a_{ij}a_{ik}\exp(x_j)\exp(x_k) 4^{\sign(\Delta_j\Delta_k)}} \\
         &\circled{3}[\leq] \frac{\beta+1}{\beta} \sum_{j=1}^{\n} \sum_{i=1}^{\m} (A\exp(x))_i^{\frac{1}{\beta}}\Delta_j^2 a_{ij} \exp(x_j) \\
         &\circled{4}[=] \frac{\beta+1}{\beta} \sum_{j=1}^{\n} \Delta_j^2 (\nabla_j \fr(x)+1) \\
         &\circled{5}[=] \sum_{j=1}^{\n} -\frac{c_j}{4}\Delta_j \truncgrad(x) (\nabla_j \fr(x)+1) \\
         &\circled{6}[\leq]  -\sum_{j=1}^{\n}\frac{c_j}{2} \Delta_j \nabla_j \fr(x). \\
     \end{aligned}
    \end{align}
    We used the inequalities $\nabla^2_{jk} \fr(x+t\Delta) \leq 2 \nabla^2_{jk} \fr(x)$ and $-\nabla^2_{jk} \fr(x+t\Delta) \leq -2^{-1}\nabla^2_{jk} \fr(x)$ in $\circled{1}$. We used Cauchy-Schwarz in $\circled{2}$ with the $\n^2$-dimensional vectors 
    \[
        (\Delta_j\Delta_k\sqrt{a_{ij}a_{ik}\exp(x_j)\exp(x_k)})_{j,k\in[n]} \ \ \text{ and } \ \ (2^{\sign(\Delta_j\Delta_k)}\sqrt{a_{ij}a_{ik}\exp(x_j)\exp(x_k)})_{j,k\in[n]},
    \] 
    in order to bound the last factor, so that the two first lines of the right hand side become proportional after bounding $4^{\sign(\Delta_j\Delta_k)} \leq 4$ in $\circled{3}$. In $\circled{3}$, we also grouped these terms. In $\circled{4}$, we used the definition of the gradient. In $\circled{5}$, we used the value of $\Delta$. Finally, $\circled{6}$ is a direct consequence of the truncated gradient definition (one can check the inequality for the three cases in $\nabla_j \fr(x)\in\{[-1, 0), [0,1], (1, \infty)\}$, while taking into account the sign of $\Delta_j$).

    Now, substituting into \eqref{eq:taylor_on_log_fr} we obtain:
    \[
        \fr(x+\Delta)-\fr(x) \leq \sum_{j=1}^{\n} \left(1-\frac{c_j}{2}\right) \Delta_j\nabla_j \fr(x).
    \]
\end{proof}

\begin{proof}\textbf{of \cref{lemma:gradient_descent_compensates_packing_regret}.}\label{proof:gradient_descent_compensates_packing_regret} \linkofproof{lemma:gradient_descent_compensates_packing_regret}
    It is enough to show that for all $i \in [\n]$ we have
    \begin{equation}\label{eq:coordinate_subproblem_of_coupling_lemma}
        \etak\nuk[k][i]( \z[k-1][i]-u_i) + \etak[k][2]\L\truncgrad[i](\x)(\x[k][i]-\y[k][i]) \leq \frac{3}{2}\etak\L \nabla_i \fr(\x)( \x[k][i]-\y[k][i] )
    \end{equation}
    because then we can conclude with
    \begin{align}
     \begin{aligned}
         \innp{\etak\nuk, \z[k-1]-u} + \etak[k][2]\L\innp{\truncgrad[ ](\x), \x-\y} &\circled{1}[\leq] \frac{3}{2}\etak \L \innp{\nabla \fr(\x), \x-\y }  \\
         &\circled{2}[\leq] 3\etak \L( \fr(\x)-\fr(\y) ),
     \end{aligned}
    \end{align}
    by adding up \eqref{eq:coordinate_subproblem_of_coupling_lemma} in $\circled{1}$ and \cref{lemma:descent_lemma_coord_descent} in $\circled{2}$. In the analysis of \eqref{eq:coordinate_subproblem_of_coupling_lemma} we exploit the simple but crucial fact that is that the gradient step for each coordinate is independent of the gradient step of other coordinates, due to the constraint set being a box. We present the rest of the proof in three cases. In the cases below, we will use $\nabla_i \fr(\x)(\x[][i]-\y[][i]) \geq 0$, cf. \cref{lemma:descent_lemma_coord_descent}. And also the fact that $\etak \leq 1/4$, as we observe in \eqref{eq:bound_by_one_fourth}.
    \begin{itemize}
        \item If $\nuk[k][i] = 0$ then $\truncgrad[i](\x) =\nabla_{i} \fr(\x) \in [-1,1]$. In such a case, we have
    \begin{align*}
       \begin{aligned}
           \etak\nuk[][i](\z[k-1][i]-u_i) + \etak[k][2]\L\truncgrad[i](\x)(\x[k][i]-\y[k][i]) &= \etak[k][2]\L\nabla_i \fr(\x)(\x[][i]-\y[][i])  \\
           & \leq \frac{3}{2}\etak\L\nabla_i \fr(\x)(\x[][i]-\y[][i]).
       \end{aligned}
    \end{align*}

\item If $\nuk[][i]>0$ and $\z[][i]> -\omega$ then the mirror descent step did not need to project along coordinate $i$, and we have $\z[][i]=\z[k-1][i] - \omega\etak $, and thus $\y[][i]=\x[][i]-\omega/\L$. In this case
    \begin{align*}
       \begin{aligned}
           \ \ &\etak\nuk[][i] (\z[k-1][i]-u_i) + \etak[k][2]\L\truncgrad[i](\x)(\x[][i]-\y[][i]) \\ 
           &\circled{1}[\leq]  \etak \nabla_{i} \fr(\x) \omega + \etak[k][2] \L \nabla_i \fr(\x)(\x[i]-\y[i]) \\
           &= \etak \L \nabla_i \fr(\x)(\x[][i]-\y[][i]) + \etak[k][2] \L \nabla_i \fr(\x)(\x[][i]-\y[][i]) \\
           & \leq \frac{3}{2}\etak\L\nabla_i \fr(\x)(\x[][i]-\y[][i]).
       \end{aligned}
    \end{align*}
            Above, we obtain $\circled{1}$ from $\z[][i] - u_{i} \leq \omega$ because $\z, u \in \B$, the fact that $\nuk[][i]$ and $\x[][i]-\y[][i]$ are positive, and $1=\truncgrad[i](\x) \leq \nabla_{i} \fr(\x)$, $0 < \nuk[][{i}] \leq \nabla_{i} \fr(\x)$.
        \item If $\nuk[][i] > 0$ and $\z[][i] = -\omega$ then
        \begin{align*}
           \begin{aligned}
               \ \ &\etak\nuk[][i] (\z[k-1][i]-u_i) + \etak[k][2]\L\truncgrad[i](\x)(\x[][i]-\y[][i]) \\ 
               &\circled{1}[\leq] \etak \nabla_i \fr(\x)(\z[k-1][i] - \z[][i]) + \etak[k][2]\L\nabla_i \fr(\x)(\x[][i]-\y[][i]) \\
               &\circled{2}[\leq] \etak[k][2] \L \nabla_i \fr(\x)(\x[][i] - \y[][i]) +  \etak[k][2]\L\nabla_{i} \fr(\x)(\x[][i]-\y[][i]) \\
           & \leq \frac{3}{2}\etak\L\nabla_i \fr(\x)(\x[][i]-\y[][i]).
           \end{aligned}
        \end{align*}
    \end{itemize}
    We have $\circled{1}$ because in this case, $u_i - \z[][i]$, $\x[][i] - \y[][i]$, $\nuk[][i]$, $\truncgrad[i](\x)$ are all $\geq 0$. We also used $0<\nuk[][i] < \nabla_{i} \fr(\x)$, $0< \truncgrad[i](\x) < \nabla_{i} \fr(\x)$. In $\circled{2}$, we used $\z[k-1][i]-\z[][i]=\etak\L(\x[][i]-\y[][i])$.
\end{proof}

}

{
\section[Missing proofs from Section \ref{sec:dual}]{Missing proofs from \cref{sec:dual}}\label{sec:app:missing_proofs_dual}

\begin{lemma}[Multiplicative Weights Lemma - Additive and Mult. Guarantee]  \label{lemma:multiplicative_weights}
    Let $\newtarget{def:losses_dual}{\lossk} \in [-\sigma, \tau]^{\tildem }$ be an sequence of $\newtarget{def:dimension_of_MW}{\tildem}$-dimensional arbitrary loss vectors, for $k\in [K]$ and $\sigma, \tau \in \R_{>0}$. Denote $W^{+} = \max\{\sigma, \tau\}, W^{-} = \min\{\sigma, \tau\}$. For a target accuracy $\delta \in (0, 2W^{-}]$, learning rate $\newtarget{def:learning_rate_MW}{\etat} =\frac{\delta}{4W^{-}} \leq \frac{1}{2}$, and initial weights $\Lambdak[1] = \ones_{\tildem} \in \simplex{\tildem}$, inducing an initial uniform distribution, run the following multiplicative weights update rule
    \[
        \newtarget{def:weights_of_MW}{\Lambdak[k+1]} \gets \Lambdak[k] \odot (\ones_{\tildem}- \frac{\etat}{W^{+}}\lossk ),
    \]
    for $k=1, \dots, K\defi\frac{8\sigma \tau \log(\tildem )}{\delta^2}$, where $\odot$ represents the coordinate-wise product. Then, for every $u \in \simplex{\tildem }$ we have
    \[
        \frac{1}{K}\sum_{k=1}^{K} \innp{\lossk , \frac{\Lambdak[k]}{\norm{\Lambdak[k]}_1}} \leq \delta + \frac{1+\signwidth\etat}{K} \sum_{k=1}^{K} \innp{\lossk , u},
    \]
     where $\newtarget{def:sign_indicator_for_max_of_width_parameters}{\signwidth}$ is $1$ if $\tau \geq \sigma$ and $-1$ otherwise.
\end{lemma}

\begin{proof}\label{proof:multiplicative_weights_lemma}
    We assume without loss of generality that for a fixed $i\in[\tildem ]$ we have $u= \ecanonical$. It is enough to prove the result in this case since the general case can be obtained as a convex combination of the resulting inequalities.

    We use the potential function $\Phi^{(k)} \defi \norm{\Lambdak[k]}_1$. On the one hand we have
    \begin{align}\label{eq:upper_bound_on_potential_Phi}
       \begin{aligned}
           \Phi^{(K+1)} &\circled{1}[=] \sum_{i=1}^{\tildem }\Lambdak[K][i] \left(1-\frac{\etat}{W^{+}}\lossk[K][i] \right) \circled{2}[=] \Phi^{(K)} - \frac{\etat}{W^{+}}\Phi^{(K)}\sum_{i=1}^{\tildem } \lossk[K][i]  \frac{\Lambdak[K][i]}{\norm{\Lambdak[K]}_1} \\
           &= \Phi^{(K)} \left( 1 - \frac{\etat}{W^{+}} \innp{\lossk[K] , \frac{\Lambdak[K]}{\norm{\Lambdak[K]}_1}}\right) \circled{3}[\leq] \Phi^{(K)}\exp\left(-\frac{\etat}{W^{+}} \innp{\lossk[K] , \frac{\Lambdak[K]}{\norm{\Lambdak[K]}_1}} \right)\\
           &\circled{4}[\leq] \Phi^{(1)}\exp\left(-\frac{\etat}{W^{+}}\sum_{k=1}^{K} \innp{\lossk , \frac{\Lambdak[k]}{\norm{\Lambdak[k]}_1}} \right) = \tildem \cdot\exp\left(-\frac{\etat}{W^{+}}\sum_{k=1}^{K} \innp{\lossk , \frac{\Lambdak[k]}{\norm{\Lambdak[k]}_1}} \right).
       \end{aligned}
    \end{align}
    Here $\circled{1}$ is due to the \MW{} update rule and $\circled{2}$  uses $\Phi^{(K)} = \norm{\Lambdak[K]}_1$. Now $\circled{3}$ uses $1-x \leq e^{-x}$, for all $x\in\R$. We recursively applied all of the previous inequalities to obtain $\circled{4}$.

    On the other hand, we can lower bound
    \begin{align} \label{eq:lower_bound_on_potential_Phi}
       \begin{aligned}
           \Phi^{(K+1)} &\circled{1}[\geq] \Lambdak[K+1][i] \circled{2}[=] \Lambdak[1][i]\prod_{k=1}^K \left( 1 - \frac{\etat}{W^{+}}\lossk[k][i]  \right) \\
           &\circled{3}[\geq] (1-\etat)^{\frac{1}{W^{+}}\sum_{\{k:\lossk[k][i] \geq 0\}} \lossk[k][i] } \cdot (1+\etat)^{\frac{1}{W^{+}}\sum_{\{k:\lossk[k][i] < 0\}} \lossk[k][i] },
       \end{aligned}
    \end{align}
    where $\circled{1}$ holds by the definition of $\Phi^{(K+1)}$ as $\norm{\Lambdak[K+1]}_1$. Here, $\circled{2}$ uses the \MW{} update rule and $\circled{3}$ is due to Bernoulli's inequality: $1+rx\geq (1+x)^r$, for $-1\leq x$, $0\leq r\leq 1$, with $(x, r) \in \{(-\etat, \lossk[k][i]/W^{+}), (\etat, -\lossk[k][i]/W^{+})\}$.

    Combining \eqref{eq:upper_bound_on_potential_Phi} and \eqref{eq:lower_bound_on_potential_Phi}, taking logarithms, and multiplying by $\frac{W^{+}}{\etat}$ we obtain the following inequality $\circled{1}$, which we further bound:
    \begin{align*}
       \begin{aligned}
           \frac{W^{+} \log(\tildem )}{\etat} - \sum_{k=1}^{K} \innp{\lossk , \frac{\Lambdak[k]}{\norm{\Lambdak[k]}_1}} &\circled{1}[\geq] \frac{1}{\etat}\log(1-\etat) \Big(\sum_{\{k:\lossk[k][i] \geq 0\}} \lossk[k][i] \Big) - \frac{1}{\etat}\log(1+\etat)\Big(\sum_{\{k:\lossk[k][i] < 0\}} \lossk[k][i] \Big) \\
           &\circled{2}[\geq] (-1 -\etat) \Big(\sum_{\{k:\lossk[k][i] \geq 0\}} \lossk[k][i] \Big) + (-1+ \etat)\Big(\sum_{\{k:\lossk[k][i] < 0\}} \lossk[k][i] \Big) \\
           &\circled{3}[\geq] -2\etat W^{-} K +  (-1 -\signwidth\etat) \sum_{k=1}^K \lossk[k][i] .
     \end{aligned}
    \end{align*}
    In $\circled{2}$, we used $\log(1-\etat) \geq -\etat - \etat^2$ and $\log(1+\etat) \geq \etat -\etat^2$ for $\etat \leq 1/2$. For $\circled{3}$, we have two cases. If $\sigma > \tau$ we have $\signwidth=-1$ and we use $-2\etat\lossk[k][i]  \geq -2\etat\tau = -2\etat W^{-}$ and bound $-\sum_{k=1}^K \ones_{\{k:\lossk[k][i] \geq 0\}} \geq -K$. Otherwise, we use $2\etat\lossk[k][i]  \geq -2\etat\sigma = -2\etat W^{-}$ and bound $-\sum_{k=1}^K \ones_{\{k:\lossk[k][i] < 0\}} \geq -K$. Reorganizing terms, dividing by $K$, and using $u = \ecanonical$ we obtain
    \[
        \frac{1}{K}\sum_{k=1}^{K} \innp{\lossk , \frac{\Lambdak[k]}{\norm{\Lambdak[k]}_1}} \leq \frac{W^{+} \log(\tildem )}{\etat K} + 2\etat W^{-} + \frac{1+\signwidth\etat}{K} \sum_{k=1}^K \innp{\lossk , u},
    \]
    and finally substituting the value of $\etat =\frac{\delta}{4W^{-}}$ and $K=\frac{8\sigma \tau \log(\tildem )}{\delta^2} =\frac{8W^{+} W^{-} \log(\tildem )}{\delta^2}$ in the statement we obtain the desired result:
    \[
        \frac{1}{K}\sum_{k=1}^{K} \innp{\lossk , \frac{\Lambdak[k]}{\norm{\Lambdak[k]}_1}} \leq \frac{\delta}{2}  + \frac{\delta}{2}+ \frac{1+\signwidth\etat}{K} \sum_{k=1}^{K} \innp{\lossk , u}.
    \]
    
\end{proof}

\begin{proof}\textbf{of \cref{lemma:PST_guarantee}.} \linkofproof{lemma:PST_guarantee}
    Since we assumed $\epsilon < 4\min\{\tau, \sigma\}$, the assumption on $\delta$ required by \cref{lemma:multiplicative_weights} is satisfied: $\delta = \epsilon/2 < 2\min\{\tau, \sigma\}$. The dimension of the statement was $\tildem=\m$, but we assume nothing on $\tildem$, so in the proof below it can be that $\tildem = \abs{\It}$, i.e., the dimension we obtain when we filter some constraints in \cref{alg:pst_mw}. We would need to substitute the instances of $A$ by $\AI[\It]$. With the parameters of the statement, we obtain the following inequality for any $u \in\simplex{\m}$:
    \begin{equation}\label{eq:aux:pst_computations}
    0 \circled{1}[\leq] \frac{1}{K}\sum_{k=1}^{K} \innp{\ones_{\m} -Ap^{(k)}, \frac{\Lambdak[k]}{\norm{\Lambdak[k]}_1}} \circled{2}[\leq] \frac{\epsilon}{2} + \frac{1+\signwidth\etat}{K}\sum_{k=1}^K \innp{\ones_{\m} - Ap^{(k)}, u},
\end{equation}
    where $\circled{1}$ is satisfied by the oracle assumption \eqref{eq:oracle_guarantee}, after using the fact that $\innp{\ones_{\m}, \Lambdak[k]/\norm{\Lambdak[k]}_1}=1$. Inequality $\circled{2}$ is the guarantee of the \MW{} algorithm, cf. \cref{lemma:multiplicative_weights}. We finally obtain the guarantee we had to prove:
\[
    -\epsilon \circled{1}[\leq] -\frac{\epsilon}{2(1+\signwidth\etat)} \circled{2}[\leq] 1 -\innp{A_i,\bar{p}}, \quad \quad \text{ for all } i \in [\m],
\]
where $\bar{p} \defi \frac{1}{K}\sum_{k=1}^K p^{(k)}$. Here $\circled{1}$ holds regardless of the value of $\signwidth$, due to the choice of $\etat \leq 1/2$. Furthermore, $\circled{2}$ is obtained by setting $u = \ecanonical$ for all $i \in [\m]$ and simplifying~\eqref{eq:aux:pst_computations}. 
\end{proof}

\subsection[Missing proofs from Section \ref{ssec:PST_oracle}]{Missing proofs from \cref{ssec:PST_oracle}}

The first lemma shows that the oracle returns a point in the lens $\lens[\v]$ efficiently.

\begin{lemma}\label{lem:lens_output}
  Let $\v\defi \ps/(1+\delta)\in \P$. The feasibility oracle returns a point in the intersection $\lens[\v] \cap \{x\in\Rp^n: \innp{\q, x}\leq 1\}$ in time $O(\n\log(\frac{\n}{(\omega-1)\delta} + \frac{\n}{\omega-1}))$.
\end{lemma}

\begin{proof}
    The output point $\o$ has to satisfy $\circled{1}:\innp{\c^{-1}(\o), \v} \leq 1$ and $\circled{2}:\innp{\c^{-1}(\v),\o} \leq 1+\omega\delta$ to be in the lens and $\circled{3}:\innp{\q, \o} \leq 1$ to be in the halfspace defined by $\q$.
  
    We note that the definition of $\c(\cdot)$ implies $\c^{-1}(\v) = (1+\delta)\s$.  Condition $\circled{1}$ is always trivially satisfied because $c^{-1}(\o)\in \D$ and $\v\in\P$ by construction.

    Now we have three cases. If $\innp{\s,\pq}\leq \frac{1+\omega\delta}{1+\delta}$ the oracle returns $\o=\pq$ and $\lambdao = \lambdaq$. This satisfies the other two conditions. Indeed, $\circled{3}$ comes from $\innp{\q,\pq}=1$ and $\circled{2}$ is satisfied because $\innp{\s,\pq}\leq\frac{1+\omega\delta}{1+\delta}$ implies $\innp{\c^{-1}(\v), \pq} = \innp{(1+\delta)\s, \pq} \leq 1+\omega\delta$. If we have $\innp{\q,\ps} \leq 1$, then the oracle returns $\o=\ps$ and $\lambdao = \lambdas$. In this case $\circled{3}$ is satisfied by construction, and $\circled{2}$ is satisfied because $\innp{\c^{-1}(\v), \ps} = \innp{(1+\delta)\s, \ps} = 1+\delta \leq 1+\omega\delta$.

    From now on we may focus on the third case, where $\innp{\s, \pq}> \frac{1+\omega\delta}{1+\delta}$ and $\innp{\q,\ps} > 1$. Let us define the functions $\pi_{\s},\pi_{\q}:(0,1) \rightarrow \Rp$ as:
  \[ \begin{array}{rl} \pi_{\s}(\mu) &= \innp{\s, \c((1-\mu)\s + \mu \q)}, \\ \pi_{\q}(\mu) &= \innp{\q, \c((1-\mu)\s + \mu \q)}. \end{array} \]
      The key observation relating these two functions is $(1-\mu)\pi_{\s}(\mu) + \mu \pi_{\q}(\mu) = 1$ for any $\mu\in (0,1)$ because $\innp{h,\c(h)}=1$ for any constraint $h \in \Rp$. So, if we find a $\mu^*\in (0,1)$ with $\pi_{\s}(\mu^*)\in (1, \frac{1+\omega\delta}{1+\delta})$ then $\o=\c((1-\mu^\ast)\s +\mu^\ast \q)$ will satisfy both $\circled{2}$, because of $\pi_{\s}(\mu^*)<(1+\omega\delta)/(1+\delta)$, and also $\circled{3}$, because if $\pi_{\s}(\mu^*)>1$ then $\pi_{\q}(\mu^*)<1$ by the observation. And we recover $\lambdao$ as $(1-\mu^*)\lambdas+\mu^*\lambdaq$.

    We intend to find such a $\mu^*$ with the bisection method. Despite $\pi_{\s}$ having a potential singularity near $\mu=1$, we will show it is regular enough to guarantee fast convergence. By the assumptions, $\lim_{\mu\rightarrow 1} \pi_{\s}(\mu)>\frac{1+\omega\delta}{1+\delta}$ and $\lim_{\mu\rightarrow 0} \pi_{\q}(\mu) >1$. Then, $\pi_{\q}(\mu)>1$ for any $\mu$ small enough, which means $\pi_{\s}(\mu)<1$ for any $\mu$ small enough by the observation. Finally, $\pi_{\s}$ is continuous, so we are able to find $\mu^*$ with $\pi_{\s}(\mu^*)\in(1, \frac{1+\omega\delta}{1+\delta})$ via the bisection method. The only remaining question is computing its running time, for which we lower bound the length of an interval in $(0,1)$ that satisfies the conditions.

    For that we will bound $\pi_{\s}'(\mu)$. Let us start with the definition of $\pi_{\s}'(\mu)/\pi_{\s}(\mu)$. Let $\pi_{\s}(\mu)_i\defi \frac{\s[i]}{\n((1-\mu)\s[i] +\mu \q[i])}$ be the $i$-th summand in the inner product of $\pi_{\s}(\mu)$. We have
    \[ 
    \frac{\pi_{\s}'(\mu)}{\pi_{\s}(\mu)} = \frac{ \sum_{i\in [\n]} \frac{\s[i] \n(\s[i]-\q[i])}{\n^2((1-\mu)\s[i] +\mu \q[i])^2}}{\sum_{i\in[\n]} \pi_{\s}(\mu)_i}= \frac{ \sum_{i\in [\n]} \pi_{\s}(\mu)_i\frac{(\s[i]-\q[i])}{(1-\mu)\s[i]+\mu \q[i]}}{\sum_{i\in[\n]} \pi_{\s}(\mu)_i}.
    \]
    We have $\pi_{\s}(\mu)_i \geq 0$ so the expression above is a weighted arithmetic mean, and its value is at most that of the maximum of the summands:
    \[
        \frac{\pi_{\s}'(\mu)}{\pi_{\s}(\mu)} \leq \max_{i\in [\n]} \frac{\s[i]-\q[i]}{(1-\mu)\s[i] +\mu \q[i]} \circled{1}[\leq] \n\max_{i\in[\n]} \pi_{\s}(\mu)_i \leq \n\sum_{i\in[\n]}  \pi_{\s}(\mu)_i\leq \n \pi_{\s}(\mu).
    \]
    We dropped $-\q[i]/((1-\mu)\s[i] +\mu \q[i])$ in $\circled{1}$ above. Hence, we have $\pi_{\s}'(\mu)\leq \n \pi_{\s}^2(\mu)$. This means the preimage of $J \defi(1,\frac{1+\omega\delta}{1+\delta})$ is an interval of length at least
    \[
        \frac{\frac{1+\omega\delta}{1+\delta}-1}{\max_{\mu:\pi_{\s}(\mu)\in J}\pi_{\s}'(\mu)} \geq \frac{\frac{1+\omega\delta}{1+\delta}-1}{\max_{\mu:\pi_{\s}(\mu)\in J}\n \pi_{\s}^2(\mu)} = \frac{\frac{1+\omega\delta}{1+\delta} -1}{\n\left(\frac{1+\omega\delta}{1+\delta}\right)^2}.
    \]
    We are interested in upper bounding the inverse of the length:
    \[
  \frac{\n\left(\frac{1+\omega\delta}{1+\delta}\right)^2}{\frac{1+\omega\delta}{1+\delta} -1} = \frac{\n(1+\omega\delta)^2}{(\omega-1)\delta(1+\delta)} = \n\left(\frac{1}{(\omega-1)\delta} + \frac{2\omega-1}{\omega-1} + \frac{\delta (\omega-1)}{1+\delta} \right)\leq \frac{\n}{(\omega-1)\delta} + \frac{4\n}{\omega-1}.
    \]

    Since the bisection method starts with an interval of length $1$ and progressively halves it every iteration, it takes at most $\log_2(\frac{\n}{(\omega-1)\delta} + \frac{4\n}{\omega-1})$ iterations to find a point with $\pi_{\s}(\mu^*)\in (1, \frac{1+\omega\delta}{1+\delta})$, and each step takes $O(\n)$ in processing time. Thus, the oracle returns a point in time $O(\n\log(\frac{\n}{(\omega-1)\delta} + \frac{\n}{\omega-1}))$.
\end{proof}

The following lemma bounds the lens by a box defined in terms of $\omega\delta$.

\begin{lemma}\label{lem:lens_box}
    If $x\in\lens[v]$, with $v\in \Rp^{\n}$, $\omega\delta>0$, then,
  \[\begin{array}{rl} x_i &\geq L_i \defi \max(0, 1-\sqrt{\omega\delta \n})v_i, \\
  x_i&\leq U_i \defi (1+\sqrt{\omega\delta \n} + \omega\delta \n)v_i. \end{array}\]
    We call the region $\prod_i [L_i, U_i]$ the \emph{bounding box} of the lens.
\end{lemma}
\begin{proof}\textbf{of \cref{lem:lens_box}.} \linkofproof{lem:lens_box}
    Recall that the lens is defined as the set of points $x\in\Rp^{\n}$ satisfying both $\innp{\c^{-1}(v), x} \leq 1+\omega\delta$ and $\innp{\c^{-1}(x),v} \leq 1$. Let us rewrite these conditions as sums:
    \[ 
    \left\{\begin{array}{rl} \innp{\c^{-1}(v),x} = \sum_{i\in[\n]} \frac{x_i}{nv_i} & \leq 1+\omega\delta, \\ \innp{\c^{-1}(x),v} = \sum_{i\in[\n]} \frac{v_i}{nx_i} & \leq 1.\end{array}\right.
    \]
  These two constraints are invariant up to multiplications of $x_i$ and $v_i$ by the same constant. Let $z_i=x_i/v_i$, and multiply both by $\n$ to get:
  \[
      \left\{\begin{array}{ll} \sum_{i\in[\n]} z_i & \leq (1+\omega\delta)\n, \\ \sum_{i\in[\n]} z_i^{-1} & \leq \n.\end{array}\right.
  \]
    It is our purpose to find the maximum and minimum of $z_i$ in this region. Because the system is symmetric under reordering of the $z_i$, we may focus on bounding $z_1$. Since the region is convex and symmetric under reordering of the variables, and because the function $z\mapsto z_1$ is symmetric under reordering of the last $(\n-1)$ variables, we may also assume that the maximum and minimum of this function are attained in points with $z_2=\dots=z_{\n}$.

    This brings us to:
    \[
        \left\{\begin{array}{lll} z_1 & +(\n-1)z_2& \leq (1+\omega\delta)\n, \\ z_1^{-1} &+ (\n-1)z_2^{-1} & \leq \n.\end{array}\right.
    \]

    The two constraints independently will never have normals vectors proportional to $e_1$. Furthermore the feasible region of the second constraint in $\Rp^{\n}$ is contained in the interior of $\Rp^{\n}$. This means that the solutions maximizing and minimizing $z_1$ must satisfy both constraints with equality.

    Solving the system of equations gives two roots. The two solutions for $z_1$ are:
    \[ z_1^+ = \frac{1}{2}(\omega\delta \n +2 +\sqrt{\omega^2\delta^2\n^2 +4\omega\delta \n -4\omega\delta}),\]
    \[ z_1^- = \frac{1}{2}(\omega\delta \n +2 -\sqrt{\omega^2\delta^2\n^2 +4\omega\delta \n -4\omega\delta}).\]

    Let us bound the smaller one first. We give two such lower bounds. The trivial one is that $z_1^->0$. This comes from the fact that the second constraint already guarantees $z_i>0$.

    The second lower bound comes from $\sqrt{a+b}\leq \sqrt{a}+\sqrt{b}$ for $a,b>0$, a consequence of the triangle inequality:

    \[ \begin{array}{rl} z_1^- &= \frac{1}{2}(\omega\delta \n +2 -\sqrt{\omega^2\delta^2\n^2 +4\omega\delta \n -4\omega\delta}) \geq \frac{1}{2}(\omega\delta \n +2 -\sqrt{\omega^2\delta^2\n^2 +4\omega\delta \n }) \geq \\
   &\geq \frac{1}{2}(\omega\delta \n +2 -\omega \delta \n - 2 \sqrt{\omega\delta \n}) \geq 1-\sqrt{\omega\delta \n}.\end{array}\]

    Now let us study the larger root. As with the other root, we remove the $-4\omega\delta$ term in the square root, then apply the triangle inequality:

    \[ \begin{array}{rl} z_1^+ &= \frac{1}{2}(\omega\delta \n +2 +\sqrt{\omega^2\delta^2\n^2 +4\omega\delta \n -4\omega\delta}) \leq \frac{1}{2}(\omega\delta \n +2 -\sqrt{\omega^2\delta^2\n^2 +4\omega\delta \n }) \leq \\
   &\leq \frac{1}{2}(\omega\delta \n +2 +\omega \delta \n + 2 \sqrt{\omega\delta \n}) \geq 1+\omega\delta \n+\sqrt{\omega\delta \n}.\end{array}\]

   Undoing the change of variables $z_i=x_i/v_i$ we obtain the desired bounds.
\end{proof}

Observe that indeed the optimum $\p[\lambdaast]$ satisfies the two conditions in the definition of $\lens[\v]$: The first condition $\innp{\c^{-1}(\p[\lambdaast]), \v} \leq 1$ comes from $\p[\lambdaast]\in \c(\D)$, so $\c^{-1}(\p[\lambdaast]) \in \D$ covers $\P$, i.e., $\innp{\c^{-1}(\p[\lambdaast]), x} \leq 1$ for all $x\in\P$. In particular it covers $\v$. The second condition $\innp{\c^{-1}(\v), \p[\lambdaast]}\leq 1+\omega\delta$ is equivalent to $\innp{\frac{1+\delta}{1+\omega\delta} s, \p[\lambdaast]} \leq 1$. Since $\omega>1$, this is satisfied as $\p[\lambdaast]\in\P$ and $s\in \D$ and therefore $\frac{1+\delta}{1+\omega\delta}s \in \D^+$. The last two lemmas provide the intuition of why the oracle returns points that are not too far from $\p[\lambdaast]$: for a fixed $\omega>1$ the bounding boxes of the respective lenses get smaller as $\delta\rightarrow 0$.

Now we can finally prove the exact guarantees of the oracle. Let $\q = A^T \lambdaq$ with $\lambdaq\in\simplex{\m}$ be the query. Furthermore, let $\s = A^T \lambdas$ with $\lambdas\in\simplex{\m}$ be our current good solution, so that the point $\ps = \c(A^T\lambdas)$ satisfies $\ghat(\ps) \leq 1+\delta$. The following proposition proves the guarantees on the oracle.

  \begin{proposition}\label{prop:our_oracle_satisfies_the_nice_properties}
    Let $U\defi\ps(1+2\delta \n + \sqrt{2\delta \n})/(1+\delta)$ be the upper-most vertex of the bounding box of the lens $\lens[v]$, as defined in \cref{lem:lens_box}. Let $\Igeneric$ be the set of non-redundant constraints, defined as $\newtarget{def:I_generic}{\Igeneric} \defi \{i \in [\m] : \innp{A_i, U} \geq 1\}$. Furthermore, let the following be the width parameters of the oracle $\oracle$ implemented in \cref{alg:oracle_3}:
    \begin{equation}\label{eq:precise_oracle_width}
        \newtarget{def:width_parameter_true_tau}{\sigmatrue} \defi \min (\sqrt{\omega\delta \n}+ \omega\delta \n,\frac{1+\omega\delta}{1+\delta} \max_{i\in[\n]} \frac{1}{\s[i]} -1 ) \quad \quad \text{ and } \quad \quad \newtarget{def:width_parameter_true_sigma}{\tautrue} \defi \min(3\sqrt{\omega\delta \n},1).
    \end{equation}
      Then, the oracle $\oracle$ in \cref{alg:oracle_3} returns a pair $\lambdao\in\simplex{\m}$, $\o\in\c(\D)$ such that
  \begin{enumerate}
    \item $\o$ satisfies $\q$, i.e., $\innp{\q, \o} \leq 1$. 
    \item If $i \in \Igeneric$, it yields $\innp{A_i, \o} \in [1-\tautrue, 1+\sigmatrue]$. That is, it is compatible with the width parameters $\sigmatrue, \tautrue$ as above, with the loss $1-\innp{A_i,\o}$ in $ [-\sigmatrue, \tautrue]$.
     \item $\o$ satisfies all redundant constraints, i.e., $\innp{A_i, \o} \leq 1$, if $i \in [\m]\setminus \Igeneric$.
  \end{enumerate}
        Besides, $\o = \c(A^T \lambdao)$, and \cref{alg:oracle_3} runs in time $O(\n\log(\frac{\n}{(\omega-1)\delta} + \frac{\n}{\omega-1}))$.
\end{proposition}

\begin{proof} The first claim is proven in \cref{lem:lens_output}.

    Let us consider $\Igeneric$ now. It is clear that the constraints in $\Igeneric$ are exactly those that do not cover the bounding box of \cref{lem:lens_box} around the lens $\lens[\v]$, with $\v\defi \ps/(1+\delta)$. This is because a positive constraint covers a box if and only if it covers the upper-most vertex $U$. Since the oracle returns a point in the lens, and hence in the box, any constraint not in $\Igeneric$ is automatically satisfied by any point returned by the oracle. This is the third claim.

    Note that the geometric meaning of the second claim is  that $\o$ is close to be lying on the hyperplanes $\innp{A_i, x} = 1$. If $i\in \Igeneric$, then $\innp{A_i, U}\geq 1$ by the definition of $\Igeneric$. Define $L$ as the lower-most vertex of the bounding box of $\lens[\v]$. Now,
  \[
    \innp{A_i,\o} \geq \min_{x\in\lens[\v]} \innp{A_i, x}  \geq \min \{\innp{A_i,x} : x\in\Rp^{\n}, x_i\in [L_i, U_i]\}.
  \]
    And the second minimum will be attained in the lower-most vertex as $A_{ij}\geq 0$ for all $j\in[\n]$. Therefore:
\begin{align*}
   \begin{aligned}
     \innp{A_i,\o} & \geq \min_{x\in\L(\v,\omega\delta)} \innp{A_i,x} \geq \innp{A_i, L} = \innp{A_i, U}\frac{\max(0,1-\sqrt{\omega\delta \n})}{1+\sqrt{\omega\delta \n} +\omega \delta \n} \\
       &\geq  \frac{\max(0,1-\sqrt{\omega\delta \n})}{1+\sqrt{\omega\delta \n} +\omega \delta \n} = 1-\min\left(1, \frac{2\sqrt{\omega\delta \n}+\omega \delta \n}{1+\sqrt{\omega\delta \n} +\omega \delta \n}\right).
   \end{aligned}
\end{align*}
  The second argument of the $\min$ is only less than $1$ whenever $\omega\delta \n <1$, so we may assume the latter in order to obtain the following bound
  \[
    \innp{A_i,\o}\geq \min_{x\in\L(\v,\omega\delta)} \innp{A_i,x} \geq 1-\min\left(1, \frac{2\sqrt{\omega\delta \n}+\omega \delta \n}{1+\sqrt{\omega\delta \n} +\omega \delta \n}\right) \geq 1-\min(1, 3\sqrt{\omega\delta \n})=1-\tautrue.
    \]
  Finally, we focus on the bound with $\sigmatrue$. As before, the maximum of $\innp{A_i,x}$ will be attained at a vertex of the box, only this time it is $U$. Now, as $\v\in \P$, we have $\innp{A_i,\v}\leq 1$. We use these two facts to conclude:
  \[
    \innp{A_i,\o} \leq \max_{x\in\L(\v,\omega\delta)} \innp{A_i,x} \leq \innp{A_i,U} = \innp{A_i, \v} (1+\sqrt{\omega\delta \n} +\omega\delta \n) \leq  1+\sqrt{\omega\delta \n} +\omega\delta \n.
  \]

  The second upper bound on $\innp{A_i,\o}$ does not come from the bounding box. We only look at the linear component of the lens, and since $A_i$ is positive, it must attain a maximum in one of the vertices of the simplex in the intersection of the hyperplane with the positive orthant. We also use that $A_{ij}\leq 1$:
  \[
    \innp{A_i,\o} \leq \max_{x\in\lens[\v]} \innp{A_i,x} \leq \max_{\substack{ x\in\Rp^{\n} \\ \innp{\s,x} \leq (1+\omega\delta)/(1+\delta)}} \innp{A_i,x} = \max_{j\in [\n]} A_{ij}\frac{1+\omega\delta}{1+\delta} \frac{1}{\s[j]} \leq  \frac{1+\omega\delta}{1+\delta}\max_{j\in [\n]} \frac{1}{\s[j]}.
    \]
    Thus, $\innp{A_i,\o} \leq 1+ \min( \sqrt{\omega\delta \n} +\omega\delta \n , \frac{1+\omega\delta}{1+\delta} \max_{j\in[\n]} \frac{1}{\s[j]} -1 ) = 1+\sigmatrue$.

  The running time and the fact $\o = \c(A^T \lambdao)$ follow from \cref{lem:lens_output}.
\end{proof}

\section{Analysis of the Yamnitsky Levin algorithm}\label{app:section:yl}

The analysis of \citep{yamnitsky1982} shows a bound on the volume ratio between two consecutive bounding simplices. Affine automorphisms preserve volume ratios, so we can assume without loss of generality that $\Delta^{(k)}\subset\Rp^{\n}$, $v^{(k)} = \zeros_{\n}$ and that the $i$-th facet incident to $v^{(k)}$ is normal to $\ecanonical[i]$. In that case, the set of constraints $\mathcal{H}_v$ consists of non-negative hyperplanes. The volume of the simplex bounded by a hyperplane $h\in\Rp^{\n}$ and the first orthant is given by $\frac{1}{\n!}\prod_{i\in[\n]} h_i^{-1}$. Thus, the volume minimization problem occurring in a stage of the \YL{} algorithm, when we fix a vertex $v^{(k)}$ is equivalent to problem \eqref{eq:dual_problem}, as the latter corresponds to the logarithm of this volume, up to an additive constant, when $h = A^T \lambda$. We assume without loss of generality that the collection of constraints $\mathcal{H}_v$ forms a matrix $A$ satisfying \eqref{eq:normalization_of_A}.

  We present, in \cref{alg:yl}, a direct translation of the \YL{} simplices algorithm \citep{yamnitsky1982} running for a stage and restricted to the positive orthant. This algorithm represents exactly what the \YL{} algorithm would do should we run it on the polyhedron $\P=\{x\in\Rp^{\n} :Ax\leq \ones_{\m}\}$ with initial simplex being the intersection of the positive orthant and the initial hyperplane $s\defi \frac{1}{\n} \sum_{j\in [\n]} A_{\ell_j}$, where $\ell_j \in \argmax_{i\in[\m]} A_{ij}$. Let $\lambdas \defi \frac{1}{\n} \sum_{j\in [\n]} \ecanonical[\ell_j] \in \simplex{\m}$, where $\ecanonical[i]$ is the $i$-th element of the canonical base, and note $s = \h[\lambdas] = A^T \lambdas$. This is a natural initial simplex and is in general a better choice than the one in the classical \YL{} algorithm. Further, in contrast, it does not depend on the bit complexity of the polyhedron. The choice of $s$ guarantees that the volume of the initial simplex is at most $\frac{\n^{\n}}{\n!}$, as for all $j\in[\n]$, we have $s_j \geq A_{\ell_j, j}/\n = 1/\n$, where the last equality holds by \eqref{eq:normalization_of_A}.

For completeness, we provide below a proof on the running time of \cref{alg:yl}, which is an adaption of the one of \citep{yamnitsky1982} to this context. Recall that a stage ends when the constraints cover the centroid of the bounding simplex $\Delta^{(k)}$, or equivalently when $\ghat(\p[\lambda^{(k)}]) \leq 1+1/\n$. \cref{alg:yl} takes $\bigotilde{\N\n^3}$ operations to run while our \cref{alg:pst_mw} takes $\bigotilde{\N \n^2}$, cf. \cref{thm:dual_guarantee} with $\epsilon=1$. Our algorithm shows how one can improve by exploiting the fact that the corner is fixed during a stage while the analysis of the \YL{} algorithm does not require to be run in stages of this kind. Both algorithms use similar initial information. \cref{alg:yl} is actually better informed, as the initial hyperplane $s$ is coordinate-wise tighter than the one we use in \cref{alg:pst_mw}. However this does not affect our bounds as in the worst case they are the same constraint.

\begin{algorithm}
    \caption{Optimization of the dual of 1-fair packing with the Yamnitsky-Levin algorithm} 
    \label{alg:yl}
\begin{algorithmic}[1]
  \REQUIRE A matrix $A\in\Rp^{\m\times \n}$ with $A_{ij}\in [0,1]$ for all $i\in [\m], j\in[\n]$.
    \vspace{0.1cm}
    \hrule
    \vspace{0.1cm}
    \State $k\gets 1$
    \State $\lambda^{(k)}\gets \lambdas \in \simplex{\m}$ 
    \WHILE { $\max_{j\in [\m]} \innp{A_j, \p[\lambda^{(k)}]}>1+\frac{1}{\n}$}
    \State $j^{(k)} \gets \argmax_{j\in [\m]} \innp{A_j, \p[\lambda^{(k)}]}$
    \State $\lambda^{(k+1)} \gets (1-\frac{1}{\n^2})\lambda^{(k)} + \frac{1}{\n^2}e_{j^{(k)}} $

    \State $k\gets k+1$
    \ENDWHILE
    \State \textbf{return} $\lambda^{(k)}, \h[\lambda^{(k)}]$
\end{algorithmic}
\end{algorithm}

\begin{lemma} \label{lem:YL-fixed}
    Let $\P$ be the polyhedron defined by $\P=\{x\in\Rp^{\n} :Ax\leq \ones_{\m}\}$ where $A \in \mathcal{M}_{\m\times \n}(\Rp)$ satisfies \eqref{eq:normalization_of_A}. Assume we run \cref{alg:yl} with the initial simplex defined above. 
  Then \cref{alg:yl} finds a feasible constraint $h\in \D$ with $\ghat(h) \leq 1+1/\n$ in $\bigotilde{\n^3}$ iterations.
\end{lemma}
\begin{proof}
  If \cref{alg:yl} stops after $K$ iterations, it must return a feasible constraint with $\ghat(\h[\lambda^{(K)}]) \leq 1+1/\n$ because of the condition on the loop. Now let us prove its bound on the running time.

    The volume of the initial simplex is at most $\n^{\n}/\n!$, as we explained in its definition. The volume of $\P$ is at least $1/\n!$ since the elements of the canonical base $\ecanonical[i]$ are in $\P$, for all $i\in [\n]$. We will show now that after each iteration of the loop, the volume is reduced by a factor of $\exp(-1/(2(\n+1)^2))$.

    Let $s^{(k)} \defi A^T\lambda^{(k)}$ be the hyperplane corresponding to $\lambda^{(k)}$ in \cref{alg:yl}, and let $h^{(k)}$ be the hyperplane $A_{j^{(k)}}$ chosen at iteration $k$. If the loop has not stopped yet, then $\innp{h^{(k)}, \c(s^{(k)})}> 1+1/\n$ by the loop condition. The volume of the simplex $\simplex{(k+1)}$ associated with $s^{(k+1)}$ is:
    \begin{equation}\label{eq:volume_formula_of_convex_combination}
    \vol(\simplex{(k+1)}) = \frac{1}{\n!}\prod_{i\in[\n]} \left(\left(1-\frac{1}{\n^2}\right)s^{(k)}_i + \frac{1}{\n^2} h^{(k)}_i\right)^{-1}.
    \end{equation}
    This expression, seen as a function of $h^{(k)}$, is strictly decreasing in each coordinate. Since $\innp{h^{(k)}, \c(s^{(k)})} > 1+1/\n$, the volume is only made larger if we assumed $\innp{h^{(k)}, \c(s^{(k)})} = 1+1/\n$ instead, by reducing each coordinate appropriately.

    Furthermore since the right hand side is log-convex in the variables $h^{(k)}$, and the set $\{h\in \Rp^{\n} : \innp{h^{(k)}, \c(s^{(k)})} = 1+1/\n \}$ is a polytope, then \cref{eq:volume_formula_of_convex_combination} must attain a maximum in one of the vertices. These vertices are the points $(\n+1)s_i^{(k)} \ecanonical[i]$ for $i\in[\n]$. We replace $h^{(k)}$ by any of them (the choice does not matter as we will soon see), and divide everything by $\vol(\Delta^{(k)})$ to get:

\begin{align*}
   \begin{aligned}
       \frac{\vol(\simplex{(k+1)})} {\vol(\simplex{(k)})} &\circled{1}[\leq] \left(1-\frac{1}{\n^2}\right)^{-(n-1)} \left( \left(1-\frac{1}{\n^2}\right) + \frac{1}{\n^2}\left(\n+1\right)\right)^{-1}= \left(1-\frac{1}{\n^2}\right)^{-(\n-1)} \left( 1+\frac{1}{\n}\right)^{-1}\\ 
       & = \left(1+\frac{1}{\n^2-1}\right)^{n-1} \left( 1-\frac{1}{\n+1}\right) \circled{2}[\leq] \exp\left(\frac{n-1}{\n^2-1}\right)\left(1-\frac{1}{\n+1}\right)\\
       &\circled{3}[\leq]  \exp\left(\frac{\n-1}{\n^2-1} - \frac{1}{\n+1} - \frac{1}{2(\n+1)^2} \right) = \exp\left(\frac{-1}{2(\n+1)^2}\right).
   \end{aligned}
\end{align*}
    In $\circled{1}$ we substituted $h^{(k)}$ by $(\n+1)s_i^{(k)} \ecanonical[i]$. The choice of the index $i\in [\n]$ is irrelevant, as all choices lead to the same bound. We also used the volume formula \eqref{eq:volume_formula_of_convex_combination}. In particular $\vol(\Delta^{(k)}) = \frac{1}{\n!}\prod_{i\in[\n]} 1/s_i^{(k)}$. In $\circled{2}$ we apply $1+x \leq e^x$ to the first multiplicand for $x=\frac{1}{\n^2-1}$. In $\circled{3}$ we used $\ln(1+x) \leq -\sum_{k=1}^2 (-x)^k/k$ for $x=-1/(\n+1)$. This inequality holds for $x\in (-1, 0)$ by the Taylor series expansion of the logarithm at $x=1$. 

    Finally, since the initial volume is at most $\n^{\n}/\n!$ and the final volume is at least $1/\n!$, this algorithm takes at most $\frac{\n\log(\n)}{1/2(\n+1)^2} \in \bigotilde{\n^3}$ iterations to finalize. Note an iteration takes $O(\N)$ operations.
\end{proof}

\end{document}